\documentclass[11pt,a4paper,final]{article}

\usepackage[hmargin={25mm,25mm},vmargin={30mm,35mm}]{geometry}
\usepackage{amssymb}
\usepackage{amsthm}
\usepackage{bm}
\usepackage{mathtools}
\usepackage{booktabs}
\usepackage{pgfplots,pgfplotstable}
\pgfplotsset{compat=1.16}
\usepackage{setspace}
\usepackage{url}
\usepackage{float}
\usepackage{subcaption}
\usepackage{enumitem}
\usepackage{eufrak}

\usepackage[colorlinks,allcolors={blue},linkcolor={black}]{hyperref}

\usepackage[natbib=true,style=numeric-comp,backend=bibtex,maxnames=99,maxcitenames=2,mincitenames=1,giveninits]{biblatex}

\numberwithin{equation}{section}
\usepackage[color,notref,notcite]{showkeys}

\usepackage[normalem]{ulem}
\newcounter{corr}
\definecolor{violet}{rgb}{0.580,0.,0.827}
\newcommand{\corr}[4][]{\typeout{Warning : a correction remains in page \thepage}
 \stepcounter{corr}
	      {\color{blue}\ifmmode\text{\,\sout{\ensuremath{#2}}\,}\else\sout{#2}\fi}
        {\color{green!50!black}#3}
        {\color{violet}#4}
}

\usepackage{authblk}
\newcommand{\email}[1]{\href{mailto:#1}{#1}}

\newtheorem{theorem}{Theorem}

\newtheorem{lemma}[theorem]{Lemma}

\newtheorem{assumption}{Assumption}

\newcommand{\bmrm}[1]{{\bm{{\rm #1}}}}
\newcommand{\bmcal}[1]{{\bm{\mathcal{#1}}}}

\newcommand{\bme}{{\bm{e}}}
\newcommand{\bmf}{{\bm{f}}}

\newcommand{\bmn}{{\bm{n}}}

\newcommand{\bmu}{{\bm{u}}}
\newcommand{\bmv}{{\bm{v}}}
\newcommand{\bmw}{{\bm{w}}}
\newcommand{\bmx}{{\bm{x}}}

\newcommand{\bmz}{{\bm{z}}}

\newcommand{\bmU}{\bm{U}}

\newcommand{\bmxi}{{\bm{\xi}}}

\newcommand{\bmpsi}{{\bm{\psi}}}

\newcommand{\bmDelta}{{\bm{\Delta}}}
\newcommand{\bmPi}{{\bm{\Pi}}}
\newcommand{\bmSigma}{{\bm{\Sigma}}}
\newcommand{\bmPsi}{{\bm{\Psi}}}

\newcommand{\calF}{\mathcal{F}}

\newcommand{\calH}{\mathcal{H}}

\newcommand{\calM}{\mathcal{M}}

\newcommand{\calP}{\mathcal{P}}

\newcommand{\calT}{\mathcal{T}}

\newcommand{\bbR}{\mathbb{R}}

\newcommand{\rma}{{\rm{a}}}
\newcommand{\rmb}{{\rm{b}}}

\newcommand{\rmi}{{\rm{i}}}

\newcommand{\rmp}{{\rm{p}}}

\newcommand{\rms}{{\rm{s}}}

\newcommand{\rmu}{{\rm{u}}}

\newcommand{\rmA}{{\rm{A}}}

\newcommand{\rmD}{{\rm{D}}}

\newcommand{\eq}{ ={}& }
\newcommand{\lea}{ \le{}& }

\newcommand{\les}{ \lesssim{}& }
\newcommand{\plus}{ &{}+ }

\newcommand{\nn}{\nonumber}
\newcommand{\nl}{\nn\\}

\newcommand{\defeq}{\vcentcolon=}

\newcommand{\SPAN}[1]{\mbox{\textrm{span}}\big\{#1\big\}}
\newcommand{\CARD}[1]{\mbox{\textrm{Card}}\big(#1\big)}

\newcommand{\h}{h}
\newcommand{\T}{{T}}
\newcommand{\F}{{F}}

\newcommand{\subh}{\h}
\newcommand{\subT}{\T}
\newcommand{\subF}{\F}

\newcommand{\Mh}[1][h]{\calM_{#1}}
\newcommand{\Th}[1][h]{\calT_{#1}}
\newcommand{\Fh}[1][h]{\calF_{#1}}

\newcommand{\Fhi}{\Fh^{\rmi}}

\newcommand{\hT}{\h_\subT}

\newcommand{\ul}[1]{\underline{#1}}
\newcommand{\ol}[1]{\overline{#1}}
\newcommand{\ulbm}[1]{\ul{\bm{#1}}}

\newcommand{\bmuT}{\bmu_\subT}
\newcommand{\bmvT}{\bmv_\subT}

\newcommand{\bmuF}{\bmu_\subF}
\newcommand{\bmvF}{\bmv_\subF}

\newcommand{\bmvh}{\bmv_\subh}

\newcommand{\ulbmuT}{\ulbm{u}_\subT}
\newcommand{\ulbmvT}{\ulbm{v}_\subT}
\newcommand{\ulbmwT}{\ulbm{w}_\subT}

\newcommand{\ulbmuh}{\ulbm{u}_\subh}
\newcommand{\ulbmvh}{\ulbm{v}_\subh}
\newcommand{\ulbmwh}{\ulbm{w}_\subh}

\newcommand{\nor}{\bmn}
\newcommand{\norT}{\nor_{\subT}}
\newcommand{\norF}{\nor_{\subF}}
\newcommand{\norTF}{\nor_{\subT\subF}}

\def\R{\bbR}

\newcommand{\POLY}[1]{\calP^{#1}}
\newcommand{\bmPOLY}[1]{\bmcal{P}^{#1}}
\newcommand{\HS}[1]{H^{#1}}
\newcommand{\HONE}{\HS{1}}
\newcommand{\HONEzr}{\HONE_0}

\newcommand{\LP}[1]{L^{#1}}
\newcommand{\LTWO}{\LP{2}}

\newcommand{\POLYPhi}[1]{\calP_{\Phi}^{#1}}
\newcommand{\POLYPhizr}[1]{\calP_{\Phi,0}^{#1}}
\newcommand{\bmPOLYPsi}[1]{\bmcal{P}_{\bmPsi}^{#1}}
\newcommand{\bmPOLYD}[1]{\bmcal{P}_{\bmDelta}^{#1}}
\newcommand{\bmPOLYn}[1]{\bmcal{P}_{\partial_{\bmn}}^{#1}}

\DeclareMathOperator*{\DIV}{div}
\DeclareMathOperator*{\CURL}{\mathbf{curl}}

\newcommand{\neumann}{\partial_{\bmn}}

\newcommand{\norm}[2][]{\|#2\|_{#1}}
\newcommand{\seminorm}[2][]{|#2|_{#1}}
\newcommand{\brac}[2][]{(#2)_{#1}}

\newcommand{\SqBrac}[2][]{\Big[#2\Big]_{#1}}

\newcommand{\Seminorm}[2][]{\Big|#2\Big|_{#1}}
\newcommand{\Brac}[2][]{\Big(#2\Big)_{#1}}

\newcommand{\SEMINORM}[2][]{\left|#2\right|_{#1}}
\newcommand{\BRAC}[2][]{\left(#2\right)_{#1}}

\newcommand{\HHOnormTPsi}[1]{\norm[1,T,\bmPsi]{#1}}

\newcommand{\HHOnormhPsi}[1]{\norm[1,h,\bmPsi]{#1}}

\newcommand{\PiTPhi}[1]{\Pi_{\subT,\Phi}^{#1}}

\newcommand{\bmPiT}[1]{\bmPi_\subT^{#1}}
\newcommand{\bmPiF}[1]{\bmPi_\subF^{#1}}

\newcommand{\bmPiTD}[1]{\bmPi_{\subT,\bmDelta}^{#1}}
\newcommand{\bmPiFn}[1]{\bmPi_{\subF,\neumann}^{#1}}

\newcommand{\SigmaT}[1]{\Sigma_\subT^{#1}}
\newcommand{\bmSigmaT}[1]{\bmSigma_\subT^{#1}}

\newcommand{\bmSigmaTPsi}[1]{\bmSigma_{\subT,\bmPsi}^{#1}}

\newcommand{\SigmaTPhi}[1]{\Sigma_{\subT,\Phi}^{#1}}
\newcommand{\SigmahPhi}[1]{\Sigma_{\subh,\Phi}^{#1}}

\newcommand{\bmrTPsi}[1]{\bmrm{r}_{\subT,\bmPsi}^{#1}}
\newcommand{\bmrhPsi}[1]{\bmrm{r}_{\subh,\bmPsi}^{#1}}

\newcommand{\DTPhi}[1]{\rmD_{\subT,\Phi}^{#1}}

\newcommand{\bmUh}[1]{\ulbm{U}_{\subh}^{#1}}
\newcommand{\bmUhk}{\bmUh{k}}

\newcommand{\bmUTPsi}[1]{\ulbm{U}_{\subT,\bmPsi}^{#1}}
\newcommand{\bmUTPsik}{\bmUTPsi{k}}

\newcommand{\bmUhPsi}[1]{\ulbm{U}_{\subh,\bmPsi}^{#1}}
\newcommand{\bmUhPsik}{\bmUhPsi{k}}

\newcommand{\bmUhPsizr}[1]{\ulbm{U}_{\subh,\bmPsi,0}^{#1}}
\newcommand{\bmUhPsizrk}{\bmUhPsizr{k}}
\newcommand{\bmXhPsiPhik}{\ulbm{X}_{\subh,\bmPsi,\Phi}^{k}}

\newcommand{\bmITPsi}[1]{\ulbm{I}_{\subT,\bmPsi}^{#1}}
\newcommand{\bmITPsik}{\bmITPsi{k}}

\newcommand{\bmIhPsi}[1]{\ulbm{I}_{\subh,\bmPsi}^{#1}}
\newcommand{\bmIhPsik}{\bmIhPsi{k}}

\def\a{\rma}
\newcommand{\aT}{\a_\subT}
\newcommand{\sT}{\rms_\subT}
\newcommand{\ah}{\a_\subh}

\newcommand{\bT}{\rmb_T}
\def\bh{\rmb_h}

\newcommand{\reverseLogLogSlopeTriangle}[5]
{
	\pgfplotsextra
	{
		\pgfkeysgetvalue{/pgfplots/xmin}{\xmin}
		\pgfkeysgetvalue{/pgfplots/xmax}{\xmax}
		\pgfkeysgetvalue{/pgfplots/ymin}{\ymin}
		\pgfkeysgetvalue{/pgfplots/ymax}{\ymax}
		
		\pgfmathsetmacro{\xArel}{#1}
		\pgfmathsetmacro{\yArel}{#3}
		\pgfmathsetmacro{\xBrel}{#1-#2}
		\pgfmathsetmacro{\yBrel}{\yArel}
		\pgfmathsetmacro{\xCrel}{\xBrel}
		\pgfmathsetmacro{\lnxB}{\xmin*(1-(#1-#2))+\xmax*(#1-#2)} % in [xmin,xmax].
		\pgfmathsetmacro{\lnxA}{\xmin*(1-#1)+\xmax*#1} % in [xmin,xmax].
		\pgfmathsetmacro{\lnyA}{\ymin*(1-#3)+\ymax*#3} % in [ymin,ymax].
		\pgfmathsetmacro{\lnyC}{\lnyA+#4*(\lnxA-\lnxB)}
		\pgfmathsetmacro{\yCrel}{\lnyC-\ymin)/(\ymax-\ymin)}
		
		\coordinate (A) at (rel axis cs:\xArel,\yArel);
		\coordinate (B) at (rel axis cs:\xBrel,\yBrel);
		\coordinate (C) at (rel axis cs:\xCrel,\yCrel);
		
		\draw[#5]   (A)-- node[pos=0.5,anchor=north] {\scriptsize{1}}
		(B)-- 
		(C) -- node[pos=0.,anchor=east] {\scriptsize{#4}}
		cycle;
	}
}

\numberwithin{equation}{section}
 \usepackage[color,notref,notcite]{showkeys}
\begin{document}	

\title{An enriched hybrid high-order method for the Stokes problem with application to flow around submerged cylinders}
\author{Liam Yemm}
\affil{School of Mathematics, Monash University, Melbourne, Australia, \email{liam.yemm@monash.edu}}
\maketitle

\begin{abstract}
  An enriched hybrid high-order method is designed for the Stokes equations of fluid flow and is fully applicable to generic curved meshes. Minimal regularity requirements of the enrichment spaces are given, and an abstract error analysis of the scheme is provided. The method achieves consistency in the enrichment space and is proven to converge optimally in energy error. The scheme is applied to 2D flow around circular cylinders, for which the local behaviour of the velocity and pressure fields are known. By enriching the local spaces with these solutions, superior numerical results near the submerged cylinders are achieved. 
\end{abstract}

\section{Introduction}\label{sec:intro}

In this paper, an enriched hybrid high-order scheme is developed for solving the Stokes equations. Such a method is particularly well-suited for situations where certain aspects of the solution can be determined, at least locally. By achieving consistency in these components of the solution, improved approximation rates can be expected.  Moreover, high-order methods require high-order regularity to achieve optimal approximation. In the case of singular solutions, conventional methods fail to accurately capture the singularity and optimal convergence is lost. By enriching the local polynomial spaces with the singular components of the solution, optimal approximation rates are recovered. One common scenario where the exact solution lacks regularity is near corners in the domain \parencite{grisvard:1985:elliptic}. 

 The enrichment of finite element schemes is typically achieved following a partition of unity approach \parencite{melenk.babuska:1996:partition, babuska.melenk:1997:partition}. A model for crack growth was described by \textcite{belytschko.black:1999:elastic}, and similarly by \textcite{moes.dolbow.ea:1999:finite}, utilising the partition of unity functions to enrich the local spaces with the asymptotic solution around the crack. There have been numerous studies of enriched finite element methods for the Stokes equations. The earliest work on this topic appears to be that of \textcite{wagner.moes.ea:2001:extended}, which focuses on rigid particles suspended in a fluid. They developed an extended finite element method (XFEM) with polynomial spaces enriched with single-particle solutions in order to better approximate the solutions to multiple-particle flows. 
%	Inspired by the partition of unity method for the XFEM \parencite{moes.dolbow.ea:1999:finite}, originally employed to model crack growth in an elastostatics problem, \cite{foucard.vernerey:2015:x} extended the XFEM approach to discretise Stokes flow around sharp corners. 
The XFEM has also been used in the study of multi-phase flows \parencite{chessa.belytschko:2003:enriched, zlotnik.dies:2009:hierarchical}, aiming to capture velocity and pressure discontinuities at the fluid interfaces. \textcite{legrain.moes.ea:2008:stability} have applied the XFEM to mixed formulations for the treatment of holes, material inclusions, and cracks in the incompressible limit.

Within the framework of polytopal methods, an enriched virtual element method (VEM) for the Poisson problem has been proposed \parencite{benvenuti.chiozzi.ea:2019:extended}, however, only the lowest order VEM was considered and no estimates of the error were given. The same authors have proposed an enriched VEM for a linear elasticity fracture problem \parencite{benvenuti.chiozzi.ea:2022:extended}, however, this work also lacks an error analysis. The work of \textcite{artioli.mascotto:2021:enrichment} designs an enriched non-conforming virtual element method (NC-VEM) for harmonic singularities arising from irregular domains in two dimensions. Moreover, the enriched NC-VEM is capable of handling highly irregular harmonic singularities, including those arising from cracked domains. However, the analysis requires a discrete trace inequality \parencite[][Equation (49)]{artioli.mascotto:2021:enrichment}, which depends on the singular function. Following this approach, the same authors have defined an enriched NC-VEM for a plane elasticity problem with corner singularities \parencite{artioli.mascotto:2023:enriched}. An extended hybrid high-order method for the Poisson problem which avoids the use of discrete inequalities dependent on the enrichment function is provided by \textcite{yemm:2022:design}. Notably, despite these advancements, no attempts have been made to design an enriched polytopal scheme for the Stokes equations.

Proposed by \textcite{di-pietro.ern:2015:linear.elasticity, di-pietro.ern.ea:2014:arbitrary}, the hybrid high-order (HHO) method is a modern polytopal method for the discretisation of elliptic partial differential equations. HHO methods allow for arbitrary order approximation rates and are constructed to better represent the underlying physics. One of their key benefits is the facilitation of static condensation of the system matrix, which leads to a substantial reduction in globally coupled degrees of freedom. An overview of the design, analysis and applications of the HHO method is provided by \textcite{di-pietro.droniou:2020:hybrid}.  

In this paper, an enriched hybrid high-order method for the Stokes equation is designed following the approach of \textcite{yemm:2022:design} for the Poisson problem. However, due to the requirement to handle an additional pressure singularity - which is coupled to the velocity singularity - and the need to define an $\inf$--$\sup$ scheme, the analysis is more involved. Given enrichment spaces satisfying Assumption \ref{assum:regularity}, an enriched method is designed and shown to be well-posed and achieve optimal error estimates in energy norm. Moreover, the method of \textcite{yemm:2023:new} of defining polynomials on curved faces is followed, leading to the scheme proposed in this paper being valid for highly generic curved meshes.

The singularities which arise near corners of the domain in the Poisson problem have zero Laplacian. As such, the requirement of \parencite[Assumption 2]{yemm:2022:design} for the singularities to have square-integrable Laplacian is satisfied trivially for such cases, and the element polynomial spaces do not require enrichment. However, due to the nature of the Stokes problem, the singularities arising from corners of the domain do not have zero Laplacian, but rather have Laplacian equal to the gradient of the pressure singularity. As such, the assumption of square-integrable Laplacian is not trivially satisfied. Indeed, it is found that Assumption \ref{assum:regularity} is only satisfied for singularities arising from convex corners. For non-convex corners a different approach is required. As such, corner singularities are not considered. Instead, the method is tested on creeping flow around circular cylinders in Section \ref{sec:cylinders}. The results show that by enriching the spaces with relevant functions describing the local behaviour of the velocity and pressure around the cylinders, the scheme performs much better than a classical HHO scheme.

The paper is organised as follows: in Section \ref{sec:model}, the Stokes problem and its enriched HHO discretisation is presented. The main results are stated in Section \ref{sec:main.results} and proven in Section \ref{sec:proofs}. The asymptotic behaviour of Stokes flow around submerged cylinders is derived in Section \ref{sec:cylinders}, and an application of the enriched HHO scheme for flow around cylinders is given in Sections \ref{sec:cylinders.testA} and \ref{sec:cylinders.testB}.

\section{Model Problem and HHO Discretisation}\label{sec:model}

Given a bounded Lipschitz domain $\Omega \subset \mathbb{R}^d$, where $d\ge 2$, the Stokes problem seeks to find the velocity field $\bmu:\Omega \to \R^d$ and the pressure $p:\Omega \to \R$ that satisfy,
\begin{subequations}\label{eq:stokes}
	\begin{alignat}{2}
		- \nu \bmDelta \bmu + \nabla p \eq \bmf&\quad&\textrm{in }\Omega, \label{eq:stokes.momentum}\\
		\nabla\cdot \bmu \eq 0&\quad&\textrm{in }\Omega, \label{eq:stokes.mass}\\
		\bmu \eq \bm{0}&\quad&\textrm{on }\partial\Omega, \label{eq:stokes.no.slip}\\
		\int_{\Omega} p \eq 0, &\quad&\label{eq:stokes.pressure}
	\end{alignat}
\end{subequations}
where $\bmf\in\LTWO(\Omega)^d$ is a prescribed volumetric force and $\nu>0$ denotes a constant viscous term. Consider the following weak formulation: find $\bmu\in\HONEzr(\Omega)^d$, $p\in\LTWO_0(\Omega)$ such that
\begin{subequations}\label{eq:stokes.weak.form}
	\begin{alignat}{2}	
		\nu\rma(\bmu, \bmv) + \rmb(\bmv, p) \eq \int_{\Omega}\bmf\cdot\bmv\quad&\forall& \bmv\in\HONEzr(\Omega)^d, \label{eq:stokes.weak.form.momentum}\\
		-\rmb(\bmu, q) \eq 0 \quad&\forall& q\in\LTWO(\Omega),\label{eq:stokes.weak.form.mass}
	\end{alignat}
\end{subequations}
where $\LTWO_0(\Omega)$ denotes the subspace of functions in $\LTWO(\Omega)$ with zero mean value, and the bilinear forms, $\rma:\HONEzr(\Omega)^d\times\HONEzr(\Omega)^d\to\R$ and $\rmb:\HONEzr(\Omega)^d\times\LTWO(\Omega)\to\R$, are defined via
\[
\rma(\bmw,\bmv) \defeq \int_{\Omega}\nabla\bmw : \nabla\bmv \quad;\quad\rmb(\bmv,q) \defeq -\int_{\Omega}(\nabla\cdot\bmv) q,
\]
where, for a differentiable vector-valued function $\bmw$, $(\nabla \bmw)_{ij} = \partial_j w_i$ denotes the total derivative.

Let \(\calH\subset(0, \infty)\) be a countable set of mesh sizes with a unique cluster point at \(0\). For each \(h\in\calH\), partition the domain \(\Omega\) into a \emph{curved} mesh \(\Mh=(\Th, \Fh)\), where $\Th$ denotes the mesh elements and $\Fh$ the mesh faces.

The mesh elements $\Th$ are a disjoint set of open, bounded, simply connected regions in $\bbR^d$ with piece-wise $C^1$ boundary $\partial T$ satisfying $\ol{\Omega} = \bigcup_{T\in\Th} \ol{T}$.

The mesh faces $\Fh$ are a disjoint set of non-intersecting finite $C^1$ manifolds which partition the mesh skeleton: $\bigcup_{T\in\Th} \partial T = \bigcup_{F\in\Fh} \ol{F}$, and for each $F\in\Fh$ there either exists two distinct elements $T_1,T_2\in\Th$ such that $F \subset \partial T_1 \cap \partial T_2$ and $F$ is called an internal face, or there exists one element $T\in\Th$ such that $F \subset \partial T \cap \partial \Omega$ and $F$ is called a boundary face. The set of interior faces are collected in the set $\Fh^i$ and boundary faces in the set $\Fh^b$.

The parameter \(h\) is given by \(h\defeq\max_{T\in\Th}\hT\) where, for \(X=T\in\Th\) or \(X=F\in\Fh\), \(h_X\) denotes the diameter of \(X\). The set of faces attached to an element \(T\in\Th\) are collected in the set \(\Fh[T]:=\{F\in\Fh:F\subset \partial T\}\) and the set of mesh elements attached to a face $F\in\Fh$ are collected in the set $\Th[F]\defeq=\{T\in\Th:F\subset \partial T\}$. The unit normal to \(F\in\Fh[T]\) pointing outside \(T\) is denoted by \(\norTF\). 

Consider the following regularity assumption on the mesh elements.

\begin{assumption}[Regular mesh sequence]\label{assum:star.shaped}	
	There exists a constant \(\varrho>0\) such that, for each \(h\in\calH\), every \(T\in\Th\) is connected by star-shaped sets with parameter \(\varrho\) (see \cite[Definition 1.41]{di-pietro.droniou:2020:hybrid}). 
\end{assumption}

Given a mesh sequence satisfying Assumption \ref{assum:star.shaped}, the following continuous trace inequality holds \parencite{droniou.yemm:2021:robust}:

\begin{lemma}[Continuous trace inequality]
	For all \(v\in H^1(T)\), it holds,
	\begin{equation}\label{small.faces:eq:continuous.trace}
		\hT\norm[\partial T]{v}^2 \lesssim \norm[T]{v}^2+\hT^2\norm[T]{\nabla v}^2, 
	\end{equation}
	where the hidden constant depends on the mesh regularity parameter $\varrho$ and the dimension $d$.
\end{lemma}

To handle the singular solutions that may arise due to the geometry of the domain or in the presence of irregular data, assume that both the velocity and pressure can be decomposed into a `regular part' and a `singular part' as follows:
\[
\bmu = \bmu_r + \bmpsi \quad;\quad p = p_r + \phi,
\]
where $\bmu_r \in \HS{k+2}(\Th)^d$ and $p_r \in \HS{k+1}(\Th)$. The singular components $\bmpsi$ and $\phi$ are members of $\bmPsi(\Th)$ and $\Phi(\Th)$, respectively, where $\bmPsi(\Th)$ and $\Phi(\Th)$ are finite dimensional enrichment spaces satisfying the following assumption.

\begin{assumption}[Assumptions on the enrichment spaces]\label{assum:regularity}
	${}$
	\begin{enumerate} 
		\item Assumptions on $\bmPsi(\Th)$: For every $T\in\Th$,
		\begin{enumerate}[label=\normalfont(\textbf{U\arabic*}),ref=\normalfont(\textbf{U\arabic*})]
			\item\label{item:U1} \(\bmPsi(T) \subset H^1(T)^d\),
			\item\label{item:U2} \(\bmDelta \bmPsi(T) \subset L^2(T)^d\),
			\item\label{item:U3} \(\nabla \bmPsi(T) \norTF \subset L^2(F)^d\).
		\end{enumerate}
		\item Assumptions on $\Phi(\Th)$: For every $T\in\Th$,
		\begin{enumerate}[label=\normalfont(\textbf{P\arabic*}),ref=\normalfont(\textbf{P\arabic*})]
			\item\label{item:P1} \(\Phi(T) \subset \HONE(T)\).
			%			\item\label{item:P2} \(\forall F\in \Fh[T]\), \(\hat{q} \in \LTWO(F)\).
		\end{enumerate}
	\end{enumerate}
\end{assumption}

\subsection{Discretisation of the Velocity}

For an integer $\ell \ge 0$, the space of scalar-valued and vector-valued polynomials on $\R^d$ of degree at most $\ell$ are denoted by $\POLY{\ell}_d$ and $\bmPOLY{\ell}_d$, respectively. Their restrictions to an element $T\in\Th$ are denoted by $\POLY{\ell}(T) = \POLY{\ell}_d|_T$ and $\bmPOLY{\ell}(T) = \bmPOLY{\ell}_d|_T$.

In order to discretise the velocity unknowns, the following spaces and operators are introduced. For each element $T \in \Th$ and an integer $k\ge 0$, define the extended polynomial space
\begin{equation}\label{def:extended.space}
	\bmPOLYPsi{k+1}(T) \defeq \bmPOLY{k+1}(T) +  \bmPsi(T) \subset \HONE(T)^d.
\end{equation}
To project onto the extended space $\bmPOLYPsi{k+1}(T)$, the extended elliptic projector $\bmSigmaTPsi{k+1}: \HONE(T)^d \to \bmPOLYPsi{k+1}(T)$ is introduced. This projector is uniquely defined such that for all $v \in \HONE(T)^d$,
\begin{equation}\label{eq:elliptic.projector}
	\int_T\nabla(\bmv-\bmSigmaTPsi{k+1} \bmv):\nabla \bmw = 0\qquad\forall\,\bmw\in\bmPOLYPsi{k+1}(T),
\end{equation}
and
\begin{equation}\label{eq:elliptic.projector.closure}
	\int_T(\bmv - \bmSigmaTPsi{k+1}\bmv) = \bm{0}.
\end{equation}

On curved faces, the space of vector-valued polynomials is defined analogously to the scalar case \parencite[Equation (2.2)]{yemm:2023:new} as
\begin{equation}\label{eq:curved.face.space}
	\bmPOLY{k}(F) \defeq \bmPOLY{0}_d|_F + (\POLY{k}_d)^{d\times d}|_F \norF,
\end{equation}
where $\norF$ is an arbitrary unit normal to the face $F$, and $(\POLY{k}_d)^{d\times d}$ denotes the space of tensor-valued polynomials on $\R^d$. The choice of unit normal $\norF$ does not affect the definition of $\bmPOLY{k}(F)$. If the face $F$ is planar, the definition \eqref{eq:curved.face.space} of face polynomials reduces to 
\[
\bmPOLY{k}(F) = \bmPOLY{k}_d|_F.
\]

The locally enriched element and face spaces are defined for each $T\in\Th$ and $F\in\Fh$ as
\[
\bmPOLYD{k}(T) \defeq \bmPOLY{k}(T) + \bmDelta \bmPsi(T) + \nabla \Phi(T),
\]
and
\[
\bmPOLYn{k}(F) \defeq \bmPOLY{k}(F) + \sum_{T\in\Th[F]}\Brac{\nabla \bmPsi(T)\norTF + \Phi(T)\norTF}.
\]
Due to the assumptions on the singular spaces, it holds that:
\[
\bmPOLYD{k}(T) \subset \LTWO(T)^d\quad\textrm{and}\quad\bmPOLYn{k}(F) \subset \LTWO(F)^d. 
\]
Define the local discrete velocity space via:
\begin{equation}\label{eq:bmUTPsi.def}
	\bmUTPsi{k} := \bmPOLYD{k}(T) \times \bigtimes_{F\in\Fh[T]}\bmPOLYn{k}(F).
\end{equation}
The local interpolator $\bmITPsik:\HONE(T)^d\to \bmUTPsi{k}$ is defined as:
\begin{equation}\label{eq:bmITPsi.def}
	\bmITPsik \bmv = 
	(\bmPiTD{k}\bmv, (\bmPiFn{k}\bmv)_{F\in\Fh[T]}),
\end{equation}
where $\bmPiTD{k}$ and $\bmPiFn{k}$ denote the $L^2$-orthogonal projectors on the spaces $\bmPOLYD{k}(T)$ and $\bmPOLYn{k}(F)$, respectively. For each element $T \in \Th$, the local reconstruction operator $\bmrTPsi{k+1}:\bmUTPsi{k}\to\bmPOLYPsi{k+1}(T)$ is introduced, which satisfies the following equations for every $\ulbmvT = (\bmvT, (\bmvF)_{F\in\Fh[T]}) \in \bmUTPsi{k}$:
\begin{subequations}\label{eq:bmrTPsi}
	\begin{equation}
		\int_T \nabla \bmrTPsi{k+1}\ulbmvT : \nabla \bmw = -\int_T \bmvT \cdot \bmDelta \bmw + \sum_{F\in\Fh[T]}\int_{F} \bmvF \cdot (\nabla \bmw) \norTF \qquad \forall \bmw \in \bmPOLYPsi{k+1}(T), \label{eq:bmrTPsi.def} 
	\end{equation}
	and
	\begin{equation}
		\int_T (\bmrTPsi{k+1}\ulbmvT - \bmvT) = \bm{0}. \label{eq:bmrTPsi.closure}
	\end{equation}
\end{subequations}
Note that the $\LTWO(T)^d$--regularity of $\bmvT$ and $\Delta \bmw$, and the $\LTWO(F)^d$--regularity of $\bmvF$ and $(\nabla \bmw) \norTF$, combined with the inclusion $\bmPOLYPsi{k+1}(T)\subset\HONE(T)^d$ ensure that $\bmrTPsi{k+1}$ is well-defined. The extended reconstruction operator $\bmrTPsi{k+1}$ is designed to mimic an integration by parts of equation \eqref{eq:elliptic.projector} so that the following commutation property holds.

\begin{lemma}\label{lem:commutation.property}
	The following identity holds true:
	\begin{equation}\label{eq:commutation.property}
		\bmrTPsi{k+1} \circ \bmITPsik = \bmSigmaTPsi{k+1}.
	\end{equation}
\end{lemma}

\begin{proof}
	Applying equation \eqref{eq:bmrTPsi.def} with $\ulbmvT = \bmITPsik \bmv$ for some $\bmv\in\HONE(T)^d$ yields the following:
	\[
	\int_T \nabla \bmrTPsi{k+1}\bmITPsi{k} \bmv : \nabla \bmw = -\int_T \bmPiTD{k}\bmv \cdot \bmDelta \bmw + \sum_{F\in\Fh[T]}\int_{F} \bmPiFn{k}\bmv \cdot (\nabla \bmw) \norTF.
	\]
	By applying the orthogonality properties of $\bmPiTD{k}$ and $\bmPiFn{k}$ together with $\bmDelta\bmPOLYPsi{k+1}(T) \subset \bmPOLYD{k}(T)$ and $\nabla\bmPOLYPsi{k+1}(T)\norTF \subset \bmPOLYn{k}(F)$, the projectors can be removed as follows,
	\begin{equation}\label{eq:commutation.property.proof.1}
		\int_T \nabla \bmrTPsi{k+1}\bmITPsik \bmv : \nabla \bmw = -\int_T \bmv \cdot \bmDelta \bmw + \sum_{F\in\Fh[T]}\int_{F} \bmv \cdot (\nabla \bmw) \norTF = \int_T \nabla \bmv : \nabla \bmw,
	\end{equation}
	where the final inequality is derived by integrating by parts.
	Substituting $\ulbmvT = \bmITPsik \bmv$ into the closure equation \eqref{eq:bmrTPsi.closure} yields
	\begin{equation}\label{eq:commutation.property.proof.2}
		\int_T (\bmrTPsi{k+1}\bmITPsik \bmv - \bmv) = \bm{0}.
	\end{equation}
	As the elliptic projector is uniquely defined, and equations \eqref{eq:commutation.property.proof.1} and \eqref{eq:commutation.property.proof.2} match \eqref{eq:elliptic.projector} and \eqref{eq:elliptic.projector.closure} identically, it follows that $\bmrTPsi{k+1}\bmITPsik \bmv = \bmSigmaTPsi{k+1}\bmv$.
\end{proof}

Due to the lack of $\HONE$--regularity of the element space $\bmPOLYD{k}(T)$, the usual discrete HHO seminorm (which involves an $\HONE$ seminorm of the element unknowns) is not well-defined. As such, the following alternative seminorm on the discrete space $\bmUTPsi{k}$ is defined: $\HHOnormTPsi{\cdot}:\bmUTPsi{k}\to\R$,
\begin{multline}\label{eq:discrete.norm}
	\HHOnormTPsi{\ulbmvT}^2 := \norm[T]{\nabla \bmrTPsi{k+1}\ulbmvT}^2 + \hT^{-2}\norm[T]{\bmvT - \bmPiTD{k}\bmrTPsi{k+1}\ulbmvT}^2  \\ + \hT^{-1}\sum_{F\in\Fh[T]}\norm[F]{\bmvF - \bmPiFn{k}\bmrTPsi{k+1}\ulbmvT}^2.
\end{multline}
Note the identities,
\[
\HHOnormTPsi{\bmITPsik\bmrTPsi{k+1}\ulbmvT}^2 = \norm[T]{\nabla \bmrTPsi{k+1}\ulbmvT}^2
\]
and
\[
\HHOnormTPsi{\ulbmvT - \bmITPsik\bmrTPsi{k+1}\ulbmvT}^2 = \hT^{-2}\norm[T]{\bmvT - \bmPiTD{k}\bmrTPsi{k+1}\ulbmvT}^2  \\ + \hT^{-1}\sum_{F\in\Fh[T]}\norm[F]{\bmvF - \bmPiFn{k}\bmrTPsi{k+1}\ulbmvT}^2,
\]
which follow readily from the property $\bmrTPsi{k+1}\bmITPsik\bmrTPsi{k+1}\ulbmvT = \bmrTPsi{k+1}\ulbmvT$ (Lemma \ref{lem:commutation.property}) and yield the following result:
\begin{equation}\label{eq:discrete.norm.identity}
	\HHOnormTPsi{\ulbmvT}^2 = \HHOnormTPsi{\bmITPsik\bmrTPsi{k+1}\ulbmvT}^2 + \HHOnormTPsi{\ulbmvT - \bmITPsik\bmrTPsi{k+1}\ulbmvT}^2.
\end{equation}

The local bilinear form $\aT:\bmUTPsi{k}\times\bmUTPsi{k}\to\R$ is defined via
\begin{equation}\label{eq:aT.def}
	\aT(\ulbmvT, \ulbmwT) \defeq \int_T \nabla \bmrTPsi{k+1}\ulbmvT:\nabla \bmrTPsi{k+1}\ulbmwT + \sT(\ulbmvT, \ulbmwT),
\end{equation}
where $\sT:\bmUTPsi{k}\times\bmUTPsi{k}\to\R$ is a stabilisation bilinear form that satisfies the following assumptions.

\begin{assumption}[Assumptions on the stabilisation term]\label{assum:stability}
	${}$
	\begin{enumerate}
		\item The bilinear form $\sT$ is symmetric positive semi-definite.
		\item For all $\ulbmvT\in\bmUTPsi{k}$
		\begin{equation}\label{eq:sT.coercvity.and.boundedness}
			%				\HHOnormTPsi{\ulbmvT - \bmITPsik\bmrTPsi{k+1}\ulbmvT}^2 \lesssim 
			\sT(\ulbmvT, \ulbmvT) \approx \HHOnormTPsi{\ulbmvT - \bmITPsik\bmrTPsi{k+1}\ulbmvT}^2.
		\end{equation}
		%			\item For all $\bmw\in\bmPOLYPsi{k+1}(T)$
		%			\begin{equation}\label{eq:sT.polynomial.consistency}
			%				\sT(\ulbmvT, \bmITPsik\bmw) = 0.
			%			\end{equation} 
	\end{enumerate}
\end{assumption}

Given a stabilisation term satisfying Assumption \ref{assum:stability}, the following lemma holds.	

\begin{lemma}
	Let $\sT:\bmUTPsi{k}\times\bmUTPsi{k}\to\R$ be a stabilisation bilinear form satisfying Assumption \ref{assum:stability}. It holds that:
	\begin{enumerate}
		\item For all $\ulbmvT\in \bmUTPsik$ and $\bmw\in\bmPOLYPsi{k+1}(T)$,
		\begin{equation}\label{eq:sT.polynomial.consistency}
			\sT(\ulbmvT, \bmITPsik \bmw) = 0.
		\end{equation} 
		\item For all $\ulbmvT, \ulbmwT \in \bmUTPsik$,
		\begin{equation}\label{eq:stab.dependency}
			\sT(\ulbmvT, \ulbmwT) = \sT(\ulbmvT - \bmITPsik\bmrTPsi{k+1}\ulbmvT, \ulbmwT - \bmITPsik\bmrTPsi{k+1}\ulbmwT).
		\end{equation} 
		\item For all $\ulbmvT\in\bmUTPsik$,
		\begin{equation}\label{eq:aT.norm.equivalence}
			%				\HHOnormTPsi{\ulbmvT}^2 \lesssim 
			\aT(\ulbmvT, \ulbmvT) \approx \HHOnormTPsi{\ulbmvT}^2.
		\end{equation}		
	\end{enumerate}
\end{lemma}

\begin{proof}
	Consider equation \eqref{eq:sT.coercvity.and.boundedness} with $\ulbmvT = \bmITPsik \bmw$ and $\bmw\in\bmPOLYPsi{k+1}(T)$:
	\[
	\sT(\bmITPsik \bmw, \bmITPsik \bmw) \approx \HHOnormTPsi{\bmITPsik \bmw - \bmITPsik\bmrTPsi{k+1}\bmITPsik w}^2 = \HHOnormTPsi{\bmITPsik \bmw - \bmITPsik \bmw}^2 = 0,
	\]
	where the second equality follows from $\bmrTPsi{k+1}\bmITPsik \bmw = \bmSigmaTPsi{k+1}\bmw =\bmw $ due to the commutation property \eqref{eq:commutation.property} and the invariance of $\bmSigmaTPsi{k+1}$ on $\bmPOLYPsi{k+1}(T)$. As $\sT$ is assumed to be bilinear and symmetric positive semi-definite, the following Cauchy--Schwarz inequality holds for all $\ulbmvT\in \bmUTPsik$ and $\bmw\in\bmPOLYPsi{k+1}(T)$:
	\[
	|\sT(\ulbmvT, \bmITPsik \bmw)| \le \sT(\ulbmvT, \ulbmvT)^\frac12\sT(\bmITPsik \bmw, \bmITPsik \bmw)^\frac12 = 0,
	\]
	thus proving \eqref{eq:sT.polynomial.consistency}. Equation \eqref{eq:stab.dependency} follows from expanding
	\[
	\sT(\ulbmvT - \bmITPsik\bmrTPsi{k+1}\ulbmvT, \ulbmwT - \bmITPsik\bmrTPsi{k+1}\ulbmwT)
	\]
	due to the bilinearity of $\sT$ and applying equation \eqref{eq:sT.polynomial.consistency}. To prove \eqref{eq:aT.norm.equivalence} consider equation \eqref{eq:discrete.norm.identity} and apply the assumption \eqref{eq:sT.coercvity.and.boundedness}:
	\begin{align*}
		\HHOnormTPsi{\ulbmvT}^2 \eq \HHOnormTPsi{\bmITPsik\bmrTPsi{k+1}\ulbmvT}^2 + \HHOnormTPsi{\ulbmvT - \bmITPsik\bmrTPsi{k+1}\ulbmvT}^2 \nl
		\approx{}& \HHOnormTPsi{\bmITPsik\bmrTPsi{k+1}\ulbmvT}^2 + \sT(\ulbmvT, \ulbmvT) \nl
		\eq \norm[T]{\nabla\bmrTPsi{k+1}\ulbmvT}^2 + \sT(\ulbmvT, \ulbmvT) = \aT(\ulbmvT, \ulbmvT),
	\end{align*}
thus completing the proof.
\end{proof}

The stabilisation bilinear form,
\begin{multline}\label{eq:sT}
	\sT(\ulbmuT, \ulbmvT) \defeq \hT^{-2}\int_T(\bmuT - \bmPiT{k}\bmrTPsi{k+1}\ulbmuT)\cdot(\bmvT - \bmPiT{k}\bmrTPsi{k+1}\ulbmvT)  \\ + \hT^{-1}\sum_{F\in\Fh[T]}\int_F(\bmuF - \bmPiF{k}\bmrTPsi{k+1}\ulbmuT)\cdot(\bmvF - \bmPiF{k}\bmrTPsi{k+1}\ulbmvT),
\end{multline}
is considered throughout the rest of this paper. It follows trivially that $\sT$ satisfies Assumption \ref{assum:stability}. Indeed,
\[
\sT(\ulbmvT, \ulbmvT) = \HHOnormTPsi{\ulbmvT - \bmITPsik\bmrTPsi{k+1}\ulbmvT}^2.
\]

\subsection{Discretisation of the Pressure}

The local discrete pressure space is defined as:
\[
\POLYPhi{k}(T) := \POLY{k}(T) + \Phi(T).
\]
The divergence reconstruction $\DTPhi{k}:\bmUTPsi{k}\to\POLYPhi{k}(T)$ is defined to satisfy
\begin{equation}\label{eq:DTPhi.def}
	\int_T \DTPhi{k} \ulbmvT q = -\int_T \bmvT \cdot \nabla q + \sum_{F\in\Fh[T]}\int_{F} \bmvF \cdot q\norT \qquad \forall q \in \POLYPhi{k}(T).
\end{equation}
Similar to the velocity reconstruction, the divergence reconstruction is designed to produce the projection of the divergence when commuted with the interpolator. This is clarified in the following lemma. 

Denote by $\PiTPhi{k}$ the $L^2$-orthogonal projector on the space $\POLYPhi{k}(T)$.
\begin{lemma}\label{lem:DT.commutation.property}
	The following identity holds true:
	\begin{equation}\label{eq:DT.commutation.property}
		\DTPhi{k} \circ \bmITPsik = \PiTPhi{k} \DIV.
	\end{equation}
\end{lemma}

\begin{proof}
	Substituting $\ulbmvT = \bmITPsik \bmv$ into equation \eqref{eq:DTPhi.def} yields
	\begin{equation}\label{eq:DT.commutation.property.proof.1}
		\int_T \DTPhi{k} \bmITPsik \bmv q = -\int_T \bmPiTD{k}\bmv \cdot \nabla q + \sum_{F\in\Fh[T]}\int_{F} \bmPiFn{k}\bmv \cdot q \norT \qquad \forall q \in \POLYPhi{k}(T).
	\end{equation}
	Analogously to the proof of Lemma \ref{lem:commutation.property}, the $L^2$-orthogonal projectors $\bmPiTD{k}$ and $\bmPiFn{k}$ may be removed due to the inclusions $\nabla\POLYPhi{k}(T) \subset \bmPOLYD{k}(T)$ and $\POLYPhi{k}(T)\norTF \subset \bmPOLYn{k}(F)$ to yield,
	\begin{equation}\label{eq:DT.commutation.property.proof.1}
		\int_T \DTPhi{k} \bmITPsik \bmv q = -\int_T \bmv \cdot \nabla q + \sum_{F\in\Fh[T]}\int_{F} \bmv \cdot q \norT = \int_T \DIV \bmv q = \int_T \PiTPhi{k} \DIV \bmv q,
	\end{equation}
	which follows by integrating by parts and introducing the projector $\PiTPhi{k}$. Rearranging and setting $q = \DTPhi{k} \ulbmvT - \PiTPhi{k} \DIV \bmv \in \POLYPhi{k}(T)$, it is inferred that
	\[
	\int_T (\DTPhi{k} \bmITPsik \bmv - \PiTPhi{k} \DIV \bmv)^2 = 0,
	\]
	as required.
\end{proof}

The discretisation of the continuous form $\rmb$ is realised through the local discrete bilinear form $\bT:\bmUTPsi{k}\times\POLYPhi{k}(T)\to\R$, defined as:
\begin{equation}\label{eq:bT.def}
	\bT(\ulbmvT, q) := -\int_T \DTPhi{k} \ulbmvT q.
\end{equation}

\section{Global Formulation and Main Results}\label{sec:main.results}	

The global discrete space is defined as
\[
\bmUhPsik \defeq \bigtimes_{T\in\Th}\bmPOLYD{k}(T) \times \bigtimes_{F\in\Fh}\bmPOLYn{k}(F),
%	 \Big\{\ulbmvh=((\bmvT)_{T\in\Th},(\bmvF)_{F\in\Fh})\,:\,\bmvT\in\bmPOLYD{k}(T)\quad\forall T\in\Th\,,
%	\bmvF\in\bmPOLYn{k}(F)\quad\forall F\in\Fh \Big\},
\]
and is equipped with the global seminorm $\HHOnormhPsi{\cdot}:\bmUhPsik\to\R$ and interpolator $\bmIhPsik:\HONE(\Omega)^d\to\bmUhk$ defined via
\begin{align*}
	\HHOnormhPsi{\ulbmvh}^2 :\eq \sum_{T\in\Th} \HHOnormTPsi{\ulbmvT}^2\qquad\forall \ulbmvh\in\bmUhk,\\
	\bmIhPsik \bmv \big|_T \eq \bmITPsik \bmv\qquad\forall \bmv\in\HONE(\Omega)^d.
\end{align*}

Similarly, for all $\ulbmvh, \ulbmwh\in\bmUhPsik$ and $q_h\in\POLYPhi{k}(\Th)$, define
\begin{equation}\label{eq:ah.and.bh.def}
	\ah(\ulbmvh, \ulbmwh) := \sum_{T\in\Th}\aT(\ulbmvT, \ulbmwT), \qquad \bh(\ulbmvh, q_h) \defeq \sum_{T\in\Th}\bT(\ulbmvT, q_T).
\end{equation}

The boundary conditions on $\bmu$ are accounted for in the following homogeneous subspace of $\bmUhPsi{k}$:
\begin{equation*}
	\bmUhPsizr{k} := \Big\{\ulbmvh=((\bmvT)_{T\in\Th},(\bmvF)_{F\in\Fh}) \in \bmUhPsi{k} : \bmvF = \bm{0}\quad\forall F \subset \partial\Omega \Big\}.
\end{equation*}
The global discrete pressure space is defined as
\begin{equation}\label{def:Pk0}
	\POLYPhizr{k}(\Th)\defeq \left\{z\in \LTWO(\Omega)\,:\,z_{|T}\in\POLYPhi{k}(T)\quad\forall T\in\Th\,,\quad\int_\Omega z=0\right\}.
\end{equation}
The following coupled space is also defined,
\begin{equation}\label{eq:coupled.space.def}
	\bmXhPsiPhik \defeq \bmUhPsizrk\times\POLYPhizr{k},
\end{equation}
and equipped with the norm,
\begin{equation}\label{eq:coupled.space.norm.def}
	\norm[X, h]{(\ulbmvh, q_h)}^2 := \nu \HHOnormhPsi{\ulbmvh}^2 + \nu^{-1}\norm[\Omega]{q_h}^2.
\end{equation}
The discrete problem then reads: find $(\ulbmuh, p_h) \in \bmXhPsiPhik$ such that
\begin{subequations}\label{eq:discrete}
	\begin{align}	
		\nu \ah(\ulbmuh, \ulbmvh) + \bh(\ulbmvh, p_h) \eq \int_\Omega \bmf \cdot \bmvh \qquad \forall \ulbmvh\in\bmUhPsizrk, \label{eq:discrete.fluid}\\
		-\bh(\ulbmuh, q_h) \eq 0 \qquad \forall q_h\in \POLYPhizr{k}(\Th) \label{eq:discrete.fluid.div}.
	\end{align}
\end{subequations}

The main results are as follows:

\begin{theorem}[Existence and uniqueness]\label{thm:existence.and.uniqueness}
	There exists a unique solution $(\ulbmuh, p_h)\in\bmXhPsiPhik$ to problem \eqref{eq:discrete}.
\end{theorem}

\begin{theorem}[Error estimate]\label{thm:error.estimate}
	The discrete solution $(\ulbmuh, p_h)\in\bmUhPsizrk\times\POLYPhizr{k}$ to problem \eqref{eq:discrete} satisfies the following error estimate:
	\begin{equation}\label{eq:error.estimate}
		\norm[X, h]{(\ulbmuh - \bmIhPsik \bmu, p_h - \SigmahPhi{k} p)} \lesssim h^{k+1}\Brac{\nu^\frac12\seminorm[\HS{k+2}(\Th)^d]{\bmu_r} + \nu^{-\frac12}\seminorm[\HS{k+1}(\Th)]{p_r}},
	\end{equation}
	where $(\bmu, p) = (\bmu_r + \bmpsi, p_r + \phi) \in \HONE(\Omega)^d \times \LTWO_0(\Omega)$ is the unique solution to problem \eqref{eq:stokes.weak.form}  satisfying the additional regularity $\bmDelta\bmu \in \LTWO(\Omega)^d$ and $p \in \HONE(\Omega)$ with $(\bmu_r, p_r)\in\HS{k+2}(\Th)^d\times\HS{k+1}(\Th)$ denoting the regular parts of the solution and $(\bmpsi, \phi)\in\bmPsi(\Th)\times\Phi(\Th)$ the singular parts of the solution.
\end{theorem}

\section{Proof of Main Results}\label{sec:proofs}	

In this section, the proofs of the main results stated in Section \ref{sec:main.results} are provided. The proof of these estimates is quite complex due to the limited regularity of the unknowns, and lack of typical functional analysis tools such as discrete inverse inequalities. Specifically, the proofs rely on the consistency of each of the discrete bilinear forms $\ah$ and $\bh$ stated in Theorems \ref{thm:ah.consistency} and \ref{thm:bh.consistency}, and the $\inf$--$\sup$ stability of the saddle-point problem \eqref{eq:discrete} as stated in Theorem \ref{thm:inf.sup.stability}. 

Many of the results in this section echo those for the regular (non-enriched) HHO discretisation of the Stokes scheme \parencite[Section 8]{di-pietro.droniou:2020:hybrid}. However, the proofs presented in this paper possess many differences. Notably, assumptions about $\HONE(T)^d$-regularity of $\bmPOLYD{k}(T)$ were not made at any stage. Consequently, gradients or Dirichlet traces of the element unknowns cannot be considered. This necessity led to the definition of the alternate discrete norm \eqref{eq:coupled.space.norm.def} under which stability and error estimates are stated. The enrichment of discrete polynomial spaces with non-polynomial functions further complicates matters, as the validity of usual discrete inverse inequalities is unknown. Additionally, the approximation properties of each extended projector in any norm other than that induced by the inner-product under which it is defined are unknown. 

\begin{lemma}
	It holds, for all $\ulbmvT\in\bmUTPsik$, that
	\begin{equation}\label{eq:discrete.terms.bound}
		\hT^{-1}\norm[T]{\bmvT - \bmrTPsi{k+1}\ulbmvT} + \hT^{-\frac12}\Brac{\sum_{F\in\Fh[T]}\norm[F]{\bmvF - \bmrTPsi{k+1}\ulbmvT}^2}^\frac12 + \sT(\ulbmvT, \ulbmvT)^\frac12 \lesssim \HHOnormTPsi{\ulbmvT}.
	\end{equation}
\end{lemma}

\begin{proof}
	The upper bound on the stabilisation term follows trivially from the assumption \eqref{eq:sT.coercvity.and.boundedness} and the identity \eqref{eq:discrete.norm.identity}. Consider,
	\begin{align}\label{eq:discrete.terms.bound.proof.1}
		\hT^{-1}\norm[T]{\bmvT - \bmrTPsi{k+1}\ulbmvT} \lea \hT^{-1}\norm[T]{\bmvT - \bmPiTD{k}\bmrTPsi{k+1}\ulbmvT} + \hT^{-1}\norm[T]{\bmrTPsi{k+1}\ulbmvT - \bmPiTD{k}\bmrTPsi{k+1}\ulbmvT} \nl \lea \hT^{-1}\norm[T]{\bmvT - \bmPiTD{k}\bmrTPsi{k+1}\ulbmvT} + \hT^{-1}\norm[T]{\bmrTPsi{k+1}\ulbmvT - \bmPiT{0}\bmrTPsi{k+1}\ulbmvT} \nl \les \hT^{-1}\norm[T]{\bmvT - \bmPiTD{k}\bmrTPsi{k+1}\ulbmvT} + \norm[T]{\nabla\bmrTPsi{k+1}\ulbmvT},
	\end{align}
	where the term $\bmPiTD{k}\bmrTPsi{k+1}\ulbmvT$ has been added and subtracted and a triangle inequality applied, and in the second line the term $\bmPiTD{k}\bmrTPsi{k+1}\ulbmvT$ is replaced with $\bmPiT{0}\bmrTPsi{k+1}\ulbmvT$ due to the optimal approximation properties of orthogonal projectors. The final inequality is a result of a Poincar\'{e}--Wirtinger inequality. Similarly, add and subtract $\bmPiFn{k}\bmrTPsi{k+1}\ulbmvT$ to the boundary term of \eqref{eq:discrete.terms.bound} and apply a triangle inequality to yield,
	\begin{align}\label{eq:discrete.terms.bound.proof.2}
		\hT^{-\frac12}&\Brac{\sum_{F\in\Fh[T]}\norm[F]{\bmvF - \bmrTPsi{k+1}\ulbmvT}^2}^\frac12 \nl \lea \hT^{-\frac12}\Brac{\sum_{F\in\Fh[T]}\norm[F]{\bmvF - \bmPiFn{k}\bmrTPsi{k+1}\ulbmvT}^2}^\frac12 + \hT^{-\frac12}\Brac{\sum_{F\in\Fh[T]}\norm[F]{\bmrTPsi{k+1}\ulbmvT - \bmPiFn{k}\bmrTPsi{k+1}\ulbmvT}^2}^\frac12 \nl \lea \hT^{-\frac12}\Brac{\sum_{F\in\Fh[T]}\norm[F]{\bmvF - \bmPiFn{k}\bmrTPsi{k+1}\ulbmvT}^2}^\frac12 + \hT^{-\frac12}\Brac{\sum_{F\in\Fh[T]}\norm[F]{\bmrTPsi{k+1}\ulbmvT - \bmPiT{0}\bmrTPsi{k+1}\ulbmvT}^2}^\frac12 \nl \les \hT^{-\frac12}\Brac{\sum_{F\in\Fh[T]}\norm[F]{\bmvF - \bmPiFn{k}\bmrTPsi{k+1}\ulbmvT}^2}^\frac12 + \norm[T]{\nabla\bmrTPsi{k+1}\ulbmvT},
	\end{align}
	where the term $\bmPiFn{k}\bmrTPsi{k+1}\ulbmvT$ is replaced with $\bmPiT{0}\bmrTPsi{k+1}\ulbmvT$ in the second inequality due to the optimal approximation properties of orthogonal projectors and the fact that $\bmPOLY{0}(T)|_F\subset \bmPOLYn{k}(F)$, and the continuous trace inequality \eqref{small.faces:eq:continuous.trace} and a Poincar\'{e}--Wirtinger inequality are applied to infer the final inequality. The result follows as each of the terms in \eqref{eq:discrete.terms.bound.proof.1} and \eqref{eq:discrete.terms.bound.proof.2} are bounded above by $\HHOnormTPsi{\ulbmvT}$.
\end{proof}

\begin{lemma}[Boundedness of the local interpolator]\label{lem:bmITPsik.bound}
	It holds, for all $\bmw\in\HONE(T)^d$, that
	\begin{equation}\label{eq:bmITPsik.bound}
		\HHOnormTPsi{\bmITPsik\bmw} \lesssim \seminorm[\HONE(T)^d]{\bmw}.
	\end{equation}
\end{lemma}

\begin{proof}
	Applying the definitions \eqref{eq:bmITPsi.def} of $\bmITPsik$ and \eqref{eq:discrete.norm} of $\HHOnormTPsi{\cdot}$ together with the commutation property \eqref{eq:commutation.property} yields,
	\begin{align*}
		\HHOnormTPsi{\bmITPsik\bmw}^2 \eq \norm[T]{\nabla \bmSigmaTPsi{k+1}\bmw}^2 + \hT^{-2}\norm[T]{\bmPiTD{k}(\bmw - \bmSigmaTPsi{k+1}\bmw)}^2 + \hT^{-1}\sum_{F\in\Fh[T]}\norm[F]{\bmPiFn{k}(\bmw - \bmSigmaTPsi{k+1}\bmw)}^2 \nl
		\lea \norm[T]{\nabla \bmw}^2 + \hT^{-2}\norm[T]{\bmw - \bmSigmaTPsi{k+1}\bmw}^2 + \hT^{-1}\sum_{F\in\Fh[T]}\norm[F]{\bmw - \bmSigmaTPsi{k+1}\bmw}^2,
	\end{align*}
	where the inequality follows from the boundedness of $\bmSigmaTPsi{k+1}$ in the $\HONE(T)$-seminorm, the $\LTWO(T)^d$-boundedness of $\bmPiTD{k}$ and the $\LTWO(F)^d$-boundedness of $\bmPiFn{k}$. Applying the continuous trace inequality \eqref{small.faces:eq:continuous.trace} and a Poincar\'{e}--Wirtinger inequality yields,
	\begin{equation}
		\HHOnormTPsi{\bmITPsik\bmw}^2 \lesssim \norm[T]{\nabla \bmw}^2 + \norm[T]{\nabla(\bmw - \bmSigmaTPsi{k+1}\bmw)}^2 \le 2\norm[T]{\nabla \bmw}^2,
	\end{equation}
	where the final inequality is a result of the minimisation properties of $\bmSigmaTPsi{k+1}\bmw$ in the $\HONE(T)$-seminorm.
\end{proof}

\begin{lemma}[Consistency of $\sT$]\label{lem:sT.consistency}
	Let $\sT:\bmUTPsik\times\bmUTPsik\to\R$ be a stabilisation bilinear form satisfying Assumption \ref{assum:stability}. Then, for all $\bmw = \bmw_r + \bmpsi\in\HONE(T)^d$ with $\bmw_r \in \HS{k+2}(T)^d$ and $\bmpsi \in \bmPsi(T)$,
	\begin{equation}\label{eq:sT.consistency}
		\sT(\bmITPsik \bmw, \bmITPsik \bmw)^\frac12 \lesssim \hT^{k+1} \seminorm[\HS{k+2}(T)^d]{\bmw_r}.
	\end{equation}
\end{lemma}

\begin{proof}
	Applying the condition \eqref{eq:sT.coercvity.and.boundedness} on $\sT$ with $\ulbmvT = \bmITPsik \bmw$ combined with the commutation property \eqref{eq:commutation.property} yields
	\begin{equation}\label{eq:sT.consistency.proof.1}
		\sT(\bmITPsik \bmw, \bmITPsik \bmw)^\frac12 \lesssim \HHOnormTPsi{\bmITPsik(\bmw - \bmSigmaTPsi{k+1}\bmw)} = \HHOnormTPsi{\bmITPsik(\bmw_r - \bmSigmaTPsi{k+1}\bmw_r)},
	\end{equation}
	which follows from the invariance of $\bmSigmaTPsi{k+1}$ on $\bmPsi(T)$ to remove the singular part of $\bmw$. Applying \eqref{eq:bmITPsik.bound} yields
	\begin{equation}\label{eq:sT.consistency.proof.2}
		\HHOnormTPsi{\bmITPsik(\bmw_r - \bmSigmaTPsi{k+1}\bmw_r)} \lesssim \seminorm[\HONE(T)^d]{\bmw_r - \bmSigmaTPsi{k+1}\bmw_r} \le \seminorm[\HONE(T)^d]{\bmw_r - \bmSigmaT{k+1}\bmw_r},
	\end{equation}
	where the second inequality holds due to the minimisation properties of $\bmSigmaTPsi{k+1}$, and $\bmSigmaT{k+1}\bmw_r \in \bmPOLYPsi{k+1}(T)$. The result follows by applying the approximation properties \parencite[Theorem 1.48]{di-pietro.droniou:2020:hybrid} of $\bmSigmaT{k+1}$.
\end{proof}

\begin{theorem}[Consistency of $\ah$]\label{thm:ah.consistency}
	The discrete bilinear form $\ah$ satisfies the following consistency: for all $\bmw = \bmw_r + \bmpsi\in\HONE(\Omega)^d$ such that $\bmDelta \bmw\in\LTWO(\Omega)^d$,  $\bmw_r \in \HS{k+2}(\Th)^d$ and $\bmpsi\in\bmPsi(\Th)$, and for all $\ulbmvh \in \bmUhPsizrk$,
	\begin{equation}\label{eq:ah.consistency}
		\left|-\int_{\Omega} \bmDelta \bmw \cdot \bmvh - \ah(\bmIhPsik \bmw, \ulbmvh)\right| \lesssim \HHOnormhPsi{\ulbmvh} h^{k+1} \seminorm[H^{k+2}(\Th)^d]{\bmw_r}. 
	\end{equation}		
\end{theorem}

\begin{proof} Applying the definition \eqref{eq:ah.and.bh.def} of $\ah$ and the commutation property \eqref{eq:commutation.property}, it holds that
	\begin{equation}\label{eq:ah.consistency.consistency.proof.1}
		\ah(\bmIhPsik \bmw,\ulbmvh) = \sum_{T\in\Th}\SqBrac{\int_T \nabla \bmSigmaTPsi{k+1} \bmw : \nabla \bmrTPsi{k+1}\ulbmvT + \sT(\bmITPsik \bmw, \ulbmvT)}.
		%	\eq \sum_{T\in\Th}\SqBrac{-\int_T \bmDelta \bmSigmaTPsi{k+1} \bmw \cdot \bmvT + \sum_{F\in\Fh[T]}\int_{F} ((\nabla \bmSigmaTPsi{k+1} \bmw) \norTF) \cdot \bmvF + \sT(\bmITPsik \bmw, \ulbmvT)}, 
	\end{equation}
	%	where in the second line we have invoked the definition \eqref{eq:bmrTPsi.def} of $\bmrTPsi{k+1}$. 
	Consider now, on each element $T\in\Th$, 
	\begin{multline}\label{eq:ah.consistency.consistency.proof.2}
		-\int_{T} \bmDelta \bmw \cdot \bmvT = -\int_{T} \bmDelta \bmw \cdot (\bmvT - \bmrTPsi{k+1}\ulbmvT) + \int_{T} \nabla \bmw : \nabla \bmrTPsi{k+1}\ulbmvT \\ - \sum_{F\in\Fh[T]}\int_F (\nabla \bmw) \norTF \cdot \bmrTPsi{k+1}\ulbmvT,
	\end{multline}
	which follows by adding and subtracting the term $\int_{T} \bmDelta \bmw \cdot \bmrTPsi{k+1}\ulbmvT$ and integrating by parts.
	By the homogeneous condition on the discrete space $\bmUhPsizrk$, and the regularity $\bmw\in\HONE(\Omega)^d$ and $\bmDelta \bmw\in\LTWO(\Omega)^d$, it holds that
	\begin{equation}\label{eq:ah.consistency.consistency.proof.3}
		0 = \sum_{F\in\Fh} \sum_{T\in\Th[F]} \int_F (\nabla \bmw) \norTF \cdot \bmvF = \sum_{T\in\Th} \sum_{F\in\Fh[T]} \int_F (\nabla \bmw) \norTF \cdot \bmvF.
	\end{equation}
	Applying both \eqref{eq:ah.consistency.consistency.proof.2} and \eqref{eq:ah.consistency.consistency.proof.3}, it follows that
	\begin{multline}\label{eq:ah.consistency.consistency.proof.4}
		-\int_{\Omega} \bmDelta \bmw \cdot \bmvh = \sum_{T\in\Th}\Big[\int_{T} \nabla \bmw : \nabla \bmrTPsi{k+1}\ulbmvT -\int_{T} \bmDelta \bmw \cdot (\bmvT - \bmrTPsi{k+1}\ulbmvT) \\ 
		+ \sum_{F\in\Fh[T]}\int_F (\nabla \bmw) \norTF \cdot (\bmvF - \bmrTPsi{k+1}\ulbmvT) \Big].
	\end{multline}
	However, consider the defining equation \eqref{eq:bmrTPsi.def} of $\bmrTPsi{k+1}$ and integrate by parts. It is apparent that
	\begin{equation}\label{eq:ah.consistency.consistency.proof.5}
		-\int_{T} \bmDelta\bmz \cdot (\bmvT - \bmrTPsi{k+1}\ulbmvT) + \sum_{F\in\Fh[T]}\int_F (\nabla \bmz) \norTF \cdot (\bmvF - \bmrTPsi{k+1}\ulbmvT) = 0,
	\end{equation}
	for all $\bmz\in\bmPOLYPsi{k+1}(T)$. Therefore, recalling that $\bmw = \bmw_r + \bmpsi = \bmw_r - \bmSigmaT{k+1}\bmw_r + \bmpsi + \bmSigmaT{k+1}\bmw_r$ with $\bmpsi|_T + \bmSigmaT{k+1}\bmw_r\in \bmPOLYPsi{k+1}(T)$, apply \eqref{eq:ah.consistency.consistency.proof.5} with $\bmz = \bmpsi + \bmSigmaT{k+1}\bmw_r$ to yield,
	%		\begin{multline}\label{eq:ah.consistency.consistency.proof.6}
		%			-\int_{T} \bmDelta \bmw \cdot (\bmvT - \bmrTPsi{k+1}\ulbmvT)
		%			+ \sum_{F\in\Fh[T]}\int_F (\nabla \bmw) \norTF \cdot (\bmvF - \bmrTPsi{k+1}\ulbmvT)
		%			\\  = -\int_{T} \bmDelta \bmw_r \cdot (\bmvT - \bmrTPsi{k+1}\ulbmvT)
		%			+ \sum_{F\in\Fh[T]}\int_F (\nabla \bmw_r) \norTF \cdot (\bmvF - \bmrTPsi{k+1}\ulbmvT).
		%		\end{multline}
	%		Moreover, applying \eqref{eq:ah.consistency.consistency.proof.5} again with $\bmz = \bmSigmaT{k+1} \bmw_r$, it holds that
	\begin{multline}\label{eq:ah.consistency.consistency.proof.7}
		-\int_{T} \bmDelta \bmw \cdot (\bmvT - \bmrTPsi{k+1}\ulbmvT)
		+ \sum_{F\in\Fh[T]}\int_F (\nabla \bmw) \norTF \cdot (\bmvF - \bmrTPsi{k+1}\ulbmvT)
		\\  = -\int_{T} \bmDelta (\bmw_r - \bmSigmaT{k+1} \bmw_r) \cdot (\bmvT - \bmrTPsi{k+1}\ulbmvT)
		+ \sum_{F\in\Fh[T]}\int_F \nabla (\bmw_r - \bmSigmaT{k+1} \bmw_r) \norTF \cdot (\bmvF - \bmrTPsi{k+1}\ulbmvT).
	\end{multline}
	Therefore, combining  \eqref{eq:ah.consistency.consistency.proof.7} with \eqref{eq:ah.consistency.consistency.proof.4} results in
	\begin{multline}\label{eq:ah.consistency.consistency.proof.8}
		-\int_{\Omega} \bmDelta \bmw \cdot \bmvh = \sum_{T\in\Th}\Big[\int_{T} \nabla \bmw : \nabla \bmrTPsi{k+1}\ulbmvT -\int_{T} \bmDelta (\bmw_r - \bmSigmaT{k+1} \bmw_r) \cdot (\bmvT - \bmrTPsi{k+1}\ulbmvT) \\ 
		+ \sum_{F\in\Fh[T]}\int_F \nabla (\bmw_r - \bmSigmaT{k+1} \bmw_r) \norTF \cdot (\bmvF - \bmrTPsi{k+1}\ulbmvT) \Big].
	\end{multline}
	Equations \eqref{eq:ah.consistency.consistency.proof.1} and \eqref{eq:ah.consistency.consistency.proof.8} are now combined to yield,
	\begin{multline}\label{eq:ah.consistency.consistency.proof.9}
		-\int_{\Omega} \bmDelta \bmw \cdot \bmvh - \ah(\bmIhPsik \bmw,\ulbmvh) 
		\\ = \sum_{T\in\Th}\Big[\int_{T} \nabla (\bmw - \bmSigmaTPsi{k+1} \bmw) : \nabla \bmrTPsi{k+1}\ulbmvT -\int_{T} \bmDelta (\bmw_r - \bmSigmaT{k+1} \bmw_r) \cdot (\bmvT - \bmrTPsi{k+1}\ulbmvT) \\
		+ \sum_{F\in\Fh[T]}\int_F \nabla (\bmw_r - \bmSigmaT{k+1} \bmw_r) \norTF \cdot (\bmvF - \bmrTPsi{k+1}\ulbmvT) - \sT(\bmITPsik \bmw,\ulbmvT)  \Big].
	\end{multline}
	It follows from the definition \eqref{eq:elliptic.projector} of the extended elliptic projector $\bmSigmaTPsi{k+1}$ that
	\[
	\int_{T} \nabla (\bmw - \bmSigmaTPsi{k+1} \bmw) : \nabla \bmrTPsi{k+1}\ulbmvT = 0.
	\]
	Substituting into \eqref{eq:ah.consistency.consistency.proof.9} and applying a triangle inequality on the sum and Cauchy--Schwarz inequalities on the integrals and $\sT$ yields,
	\begin{multline}\label{eq:ah.consistency.consistency.proof.10}
		\Big|-\int_{\Omega} \bmDelta \bmw \cdot \bmvh - \ah(\bmIhPsik \bmw,\ulbmvh) \Big|
		\\ \le \sum_{T\in\Th}\Big[\hT\norm[T]{\bmDelta (\bmw_r - \bmSigmaT{k+1} \bmw_r)}\hT^{-1}\norm[T]{\bmvT - \bmrTPsi{k+1}\ulbmvT} + \sT(\bmITPsik \bmw, \bmITPsik \bmw)^\frac12 \sT(\ulbmvT, \ulbmvT)^\frac12 \\
		+ \sum_{F\in\Fh[T]}\hT^\frac12\norm[F]{\nabla (\bmw_r - \bmSigmaT{k+1} \bmw_r)\norTF}\hT^{-\frac12}\norm[F]{\bmvF - \bmrTPsi{k+1}\ulbmvT}   \Big].
	\end{multline}
	By applying a Cauchy--Schwarz on the sum over the faces and equation \eqref{eq:discrete.terms.bound}, it is inferred that
	\begin{multline}\label{eq:ah.consistency.consistency.proof.11}
		\Big|-\int_{\Omega} \bmDelta \bmw \cdot \bmvh - \ah(\bmIhPsik \bmw,\ulbmvh) \Big| \lesssim \sum_{T\in\Th}\HHOnormTPsi{\ulbmvT}\Big[\hT\norm[T]{\bmDelta (\bmw_r - \bmSigmaT{k+1} \bmw_r)}  \\ + \sT(\bmITPsik \bmw, \bmITPsik \bmw)^\frac12
		+ \Brac{\sum_{F\in\Fh[T]}\hT\norm[F]{\nabla (\bmw_r - \bmSigmaT{k+1} \bmw_r)\norTF}^2}^\frac12\Big].
	\end{multline}
	Finally, the result follows by applying the consistency \eqref{eq:sT.consistency} of $\sT$, the continuous trace inequality \eqref{small.faces:eq:continuous.trace} on the boundary term, the approximation properties \parencite[Theorem 1.48]{di-pietro.droniou:2020:hybrid} of the elliptic projector $\bmSigmaT{k+1}$ together with the assumed $\HS{k+2}$-regularity of $\bmw_r$, and a Cauchy--Schwarz inequality on the sum.
\end{proof}

\begin{theorem}[Consistency of $\bh$]\label{thm:bh.consistency}
	The discrete bilinear form $\bh$ satisfies the following consistency properties: 	
	\begin{enumerate}
		\item For all $\bmv\in\HONE(\Omega)$ and for all $q_h\in\POLYPhizr{k}(\Th)$,
		\begin{equation}\label{eq:bh.consistency.1}
			\rmb(\bmv, q_h) = \bh(\bmIhPsik \bmv, q_h). 
		\end{equation}	
		\item For all $q = q_r + \phi\in\HONE(\Omega)$ such that $q_r \in \HS{k+1}(\Th)$ and $\phi\in\Phi(\Th)$, and for all $\ulbmvh \in \bmUhPsizrk$,
		\begin{equation}\label{eq:bh.consistency.2}
			\left|\int_{\Omega} \nabla q \cdot \bmvh - \bh(\ulbmvh, \SigmahPhi{k} q)\right| \lesssim \HHOnormhPsi{\ulbmvh} h^{k+1} \seminorm[H^{k+1}(\Th)]{q_r}, 
		\end{equation}	
		where $\SigmahPhi{k}|_T=\SigmaTPhi{k}$ for all $T\in\Th$, and $\SigmaTPhi{k}$ denotes the elliptic projector on $\POLYPhi{k}(T)$.
	\end{enumerate}	
\end{theorem}

\begin{proof}
	It follows from the commutation property \eqref{eq:DT.commutation.property} and the definition \eqref{eq:ah.and.bh.def} of $\bh$ that
	\begin{multline*}
		\bh(\bmIhPsik \bmv, q_h) = \sum_{T\in\Th}-\int_{T} \PiTPhi{k} (\nabla \cdot \bmv) q_T = \sum_{T\in\Th}-\int_{T} (\nabla \cdot \bmv)q_T = -\int_{\Omega}(\nabla \cdot \bmv) q_h \\ = \rmb(\bmv, q_h).
	\end{multline*}
	Turn now to equation \eqref{eq:bh.consistency.2}. By the assumed $\HONE(\Omega)$-regularity of $q$ and the homogeneous condition on the discrete space, it holds that
	\[
	\sum_{T\in\Th}\sum_{F\in\Fh[T]} \int_F \bmvF \cdot q \norTF = \sum_{F\in\Fhi}\sum_{T\in\Th[F]} \int_F \bmvF \cdot q \norTF = 0,
	\]
	and thus
	\[
	\int_{\Omega} \nabla q \cdot \bmvh = \sum_{T\in\Th[F]}\SqBrac{\int_T \bmvT \cdot \nabla q - \sum_{F\in\Fh[T]}\int_F \bmvF \cdot q \norTF}.
	\]
	It follows from the definition \eqref{eq:bT.def} of $\bh$ that
	\[
	\bh(\ulbmvh, \SigmahPhi{k} q) = \sum_{T\in\Th}\SqBrac{\int_T \bmvT \cdot \nabla \SigmaTPhi{k} q - \sum_{F\in\Fh[T]}\int_F \bmvF \cdot \SigmaTPhi{k} q \norTF}.
	\]
	Therefore,
	\begin{multline}\label{eq:bh.consistency.proof.1}
		\int_{\Omega} \nabla q \cdot \bmvh - \bh(\ulbmvh, \SigmahPhi{k} q) \\
		= \sum_{T\in\Th}\SqBrac{\int_T \bmvT \cdot \nabla (q - \SigmaTPhi{k} q) - \sum_{F\in\Fh[T]}\int_F \bmvF \cdot (q - \SigmaTPhi{k} q) \norTF}.
	\end{multline}
	By adding and subtracting the term $\bmrTPsi{k+1}\ulbmvT$ to each of the integrals and integrating by parts it follows that	
	\begin{align}\label{eq:bh.consistency.proof.2}
		\int_{\Omega} & \nabla q \cdot \bmvh - \bh(\ulbmvh, \SigmahPhi{k} q) \nl
		\eq \sum_{T\in\Th}\Big[\int_T (\bmvT - \bmrTPsi{k+1}\ulbmvT) \cdot \nabla (q - \SigmaTPhi{k} q) - \sum_{F\in\Fh[T]}\int_F (\bmvF - \bmrTPsi{k+1}\ulbmvT) \cdot (q - \SigmaTPhi{k} q) \norTF \nl
		\plus \int_T \bmrTPsi{k+1}\ulbmvT \cdot \nabla (q - \SigmaTPhi{k} q) - \sum_{F\in\Fh[T]}\int_F \bmrTPsi{k+1}\ulbmvT \cdot (q - \SigmaTPhi{k} q) \norTF\Big]\nl
		\eq \sum_{T\in\Th}\Big[\int_T (\bmvT - \bmrTPsi{k+1}\ulbmvT) \cdot \nabla (q - \SigmaTPhi{k} q) - \sum_{F\in\Fh[T]}\int_F (\bmvF - \bmrTPsi{k+1}\ulbmvT) \cdot (q - \SigmaTPhi{k} q) \norTF \nl
		\plus \int_T \nabla \cdot \bmrTPsi{k+1}\ulbmvT \cdot (q - \SigmaTPhi{k} q)\Big].
	\end{align}
	Applying a triangle inequality on the sum and Cauchy--Schwarz inequalities on the integrals yields,
	\begin{align}
		\int_{\Omega} & \nabla q \cdot \bmvh - \bh(\ulbmvh, \SigmahPhi{k} q) \nl
		\les \sum_{T\in\Th}\Big[\hT^{-1}\norm[T]{\bmvT - \bmrTPsi{k+1}\ulbmvT} \hT\norm[T]{ \nabla (q - \SigmaTPhi{k} q)} + \norm[T]{\nabla \cdot \bmrTPsi{k+1}\ulbmvT} \norm[T]{q - \SigmaTPhi{k} q} \nl \plus \hT^{-\frac12}\Brac{\sum_{F\in\Fh[T]}\norm[F]{\bmvF - \bmrTPsi{k+1}\ulbmvT}^2}^\frac12 \hT^{\frac12}\Brac{\sum_{F\in\Fh[T]}\norm[F]{q - \SigmaTPhi{k} q}^2}^\frac12 
		\Big].
	\end{align}
	By invoking equation \eqref{eq:discrete.terms.bound} together with the bound,
	\[
	\norm[T]{\nabla \cdot \bmrTPsi{k+1}\ulbmvT} \lesssim \norm[T]{\nabla \bmrTPsi{k+1}\ulbmvT} \le \HHOnormTPsi{\ulbmvT},
	\]
	it is inferred that
	\begin{align}
		\int_{\Omega} & \nabla q \cdot \bmvh - \bh(\ulbmvh, \SigmahPhi{k} q) \nl
		\les \sum_{T\in\Th}\HHOnormTPsi{\ulbmvT}\Brac{\hT\norm[T]{ \nabla (q - \SigmaTPhi{k} q)} + \norm[T]{q - \SigmaTPhi{k} q}  + \hT^{\frac12}\Brac{\sum_{F\in\Fh[T]}\norm[F]{q - \SigmaTPhi{k} q}^2}^\frac12}.
	\end{align}
	Applying the continuous trace inequality \eqref{small.faces:eq:continuous.trace} and a Poincar\'{e}--Wirtinger inequality yields,
	\begin{align}
		\int_{\Omega} \nabla q \cdot \bmvh - \bh(\ulbmvh, \SigmahPhi{k} q)
		\les \sum_{T\in\Th}\HHOnormTPsi{\ulbmvT}\hT\norm[T]{ \nabla (q - \SigmaTPhi{k} q)} \nl
		\eq \sum_{T\in\Th}\HHOnormTPsi{\ulbmvT}\hT\norm[T]{ \nabla (q_r - \SigmaTPhi{k} q_r)}\nl
		\lea \sum_{T\in\Th}\HHOnormTPsi{\ulbmvT}\hT\norm[T]{ \nabla (q_r - \SigmaT{k} q_r)},
	\end{align}
	where the invariance of $\SigmaTPhi{k}$ on $\Phi(T)$ has been invoked to remove the singular part $\phi$, and in the final inequality the projector $\SigmaTPhi{k}$ is replaced with $\SigmaT{k}$ due to the former projector minimising the $\HONE(T)$-seminorm. The result follows by applying the approximation properties \parencite[Theorem 1.48]{di-pietro.droniou:2020:hybrid} of $\SigmaT{k}$ and invoking a Cauchy--Schwarz inequality on the sum.
\end{proof}

\begin{theorem}[$\inf$--$\sup$ stability]\label{thm:inf.sup.stability}
	For all $(\ulbmwh, z_h)\in\bmXhPsiPhik$, the following $\inf$--$\sup$ stability condition holds:
	\begin{equation}\label{eq:inf.sup.stability}
		\norm[X,h]{(\ulbmwh, z_h)} \lesssim \sup_{(\ulbmvh, q_h) \in \bmXhPsiPhik : \norm[X,h]{(\ulbmvh, q_h)} = 1} \brac{\nu\ah(\ulbmwh, \ulbmvh) + \bh(\ulbmvh, z_h) - \bh(\ulbmwh, q_h)}. 
	\end{equation}		
\end{theorem}

\begin{proof}
	It follows trivially from the definition \eqref{eq:ah.and.bh.def} of $\ah$ and equation \eqref{eq:aT.norm.equivalence} that there exists a $C_{\rma} > 0$ independent of $h$ and $\nu$ such that
	\begin{equation}\label{eq:inf.sup.stability.proof.1}
		C_{\rma}^{-1}\nu\HHOnormhPsi{\ulbmwh}^2 \le \nu \ah(\ulbmwh, \ulbmwh) \le C_{\rma} \nu\HHOnormhPsi{\ulbmwh}^2.
	\end{equation}
	Consider now the stability of $\bh$. It is a classical result in functional analysis that for all $z\in\LTWO(\Omega)$ there exists a $\bmxi \in \HONEzr(\Omega)^d$ such that $-\nabla \cdot \bmxi = z$ and $\norm[\HONE(\Omega)^d]{\bmxi} \lesssim \norm[\Omega]{z}$ \parencite[cf.][Lemma 8.3]{di-pietro.droniou:2020:hybrid}. Choosing a $\bmxi \in \HONEzr(\Omega)^d$ such that $-\nabla \cdot \bmxi = z_h\in\POLYPhi{k}(\Th)$ and $\norm[\HONE(\Omega)^d]{\bmxi} \lesssim \norm[\Omega]{z_h}$, it holds that
	\begin{equation}\label{eq:inf.sup.stability.proof.2}
		\norm[\Omega]{z_h}^2 = -\int_{\Omega}(\nabla \cdot \bmxi)z_h = \rmb(\bmxi, z_h) = \bh(\bmIhPsik\bmxi, z_h),
	\end{equation}
	where equation \eqref{eq:bh.consistency.1} is invoked to yield the final equality. It is clear that
	\begin{equation}\label{eq:inf.sup.stability.proof.3}
		\bh(\bmIhPsik\bmxi, z_h) = \frac{\bh(\bmIhPsik\bmxi, z_h)}{\HHOnormhPsi{\bmIhPsik\bmxi}}\HHOnormhPsi{\bmIhPsik\bmxi} \le \sup_{\ulbmvh \in \bmUhPsik} \frac{|\bh(\ulbmvh, z_h)|}{\HHOnormhPsi{\ulbmvh}}\HHOnormhPsi{\bmIhPsik\bmxi}.
	\end{equation}
	Apply the interpolant bound \eqref{eq:bmITPsik.bound} to equation \eqref{eq:inf.sup.stability.proof.3} and use the a priori condition $\norm[\HONE(\Omega)^d]{\bmxi} \lesssim \norm[\Omega]{z_h}$ to write
	\begin{equation}\label{eq:inf.sup.stability.proof.4}
		\bh(\bmIhPsik\bmxi, z_h) \lesssim \sup_{\ulbmvh \in \bmUhPsik} \frac{|\bh(\ulbmvh, z_h)|}{\HHOnormhPsi{\ulbmvh}}\norm[\HONE(\Omega)^d]{\bmxi} \lesssim \sup_{\ulbmvh \in \bmUhPsik} \frac{|\bh(\ulbmvh, z_h)|}{\HHOnormhPsi{\ulbmvh}} \norm[\Omega]{z_h}.
	\end{equation}
	Combining \eqref{eq:inf.sup.stability.proof.4} and \eqref{eq:inf.sup.stability.proof.2} readily infers the existence of a $C_{\rmb}>0$ which is independent of $\nu$ and $h$ such that	
	\begin{equation}\label{eq:inf.sup.stability.proof.5}
		C_{\rmb}^{-1}\nu^{-\frac12}\norm[\Omega]{z_h} \le \sup_{\ulbmvh \in \bmUhPsik} \frac{|\bh(\ulbmvh, z_h)|}{\nu^{\frac12}\HHOnormhPsi{\ulbmvh}}.
	\end{equation}
			Consider now the saddle-point stability condition given by \textcite[Lemma A.11]{di-pietro.droniou:2020:hybrid} which infers the existence of a constant,
			\begin{equation*}
					C_{\rmA} = \Brac{C_{\rma}^{2}(1 + 2 C_{\rmb}^{2}C_{\rma}^2)^2 + 4  C_{\rmb}^{2}}^\frac12,
				\end{equation*}
			such that
			\begin{equation*}
					C_{\rmA}^{-1} \norm[X,h]{(\ulbmwh, z_h)} \le \sup_{(\ulbmvh, q_h) \in \bmXhPsiPhik : \norm[X,h]{(\ulbmvh, q_h)} = 1}|\nu \ah(\ulbmwh, \ulbmvh) + \bh(\ulbmvh, z_h) - \bh(\ulbmwh, q_h)|. \qedhere
				\end{equation*}
\end{proof}

\begin{proof}[Proof of Theorem \ref{thm:existence.and.uniqueness}]
	The existence and uniqueness of a solution to problem \eqref{eq:discrete} follows from the $\inf$--$\sup$ stability of the scheme \parencite[Theorem 2.34]{ern.guermond:2004:theory}.
\end{proof}

\begin{proof}[Proof of Theorem \ref{thm:error.estimate}]
	Let $(\ulbmuh, p_h) \in \bmXhPsiPhik$ be the solution to the discrete problem \eqref{eq:discrete} and $(\bmu, p) \in \HONE(\Omega)^d \times \LTWO_0(\Omega)$ be the unique solution to problem \eqref{eq:stokes} satisfying the additional regularity $\bmDelta\bmu \in \LTWO(\Omega)^d$ and $p \in \HONE(\Omega)$ with $(\bmu_r, p_r)\in\HS{k+2}(\Th)^d\times\HS{k+1}(\Th)$ denoting the regular parts of the solution and $(\bmpsi, \phi)\in\bmPsi(\Th)\times\Phi(\Th)$ the singular parts of the solution. By invoking the Third Strang Lemma \parencite[Lemma A.7]{di-pietro.droniou:2020:hybrid} and the consistency results \eqref{eq:ah.consistency}, \eqref{eq:bh.consistency.1} and \eqref{eq:bh.consistency.2}, it holds that
	%		
	%		
	%		It holds for all $(\ulbmvh, q_h)\in \bmXhPsiPhik$ that
	%		\begin{multline*}
		%			\nu \ah(\ulbmuh - \bmIhPsik \bmu, \ulbmvh) + \bh(\ulbmvh, p_h - \SigmahPhi{k} p) - \bh(\ulbmuh - \bmIhPsik \bmu, q_h) 
		%			\nl
		%			= \int_{\Omega} \bmf \cdot \bmvh - \nu\ah(\bmIhPsik \bmu, \ulbmvh) -  \bh(\ulbmvh, \SigmahPhi{k} p) + \bh(\bmIhPsik \bmu, q_h),
		%		\end{multline*}
	%		where the bilinear forms have been expanded due to $(\ulbmuh, p_h)$ solving the discrete problem \eqref{eq:discrete} to yield the equality in the second line. Writing $\bmf = -\nu \bmDelta \bmu + \nabla p$ and rearranging, it holds that
	%		\begin{multline*}
		%			\nu \ah(\ulbmuh - \bmIhPsik \bmu, \ulbmvh) + \bh(\ulbmvh, p_h - \SigmahPhi{k} p) - \bh(\ulbmuh - \bmIhPsik \bmu, q_h) 
		%			\nl
		%			= \nu\Brac{-\int_{\Omega} \bmDelta\bmu \cdot \bmvh - \ah(\bmIhPsik \bmu, \ulbmvh)} + \Brac{\int_{\Omega} \nabla p \cdot \bmvh -  \bh(\ulbmvh, \SigmahPhi{k} p)} \nl - \Brac{\rmb(\bmu, q_h) - \bh(\bmIhPsik \bmu, q_h)},
		%		\end{multline*}
	%		where the addition of the term $-\rmb(\bmu, q_h)=0$ is a result of equation \eqref{eq:stokes.mass}. Taking the absolute value, applying a triangle inequality, and invoking the consistency results \eqref{eq:ah.consistency}, \eqref{eq:bh.consistency.1} and \eqref{eq:bh.consistency.2}, yields
	\begin{multline}\label{eq:error.estimate.proof.1}
		\nu \ah(\ulbmuh - \bmIhPsik \bmu, \ulbmvh) + \bh(\ulbmvh, p_h - \SigmahPhi{k} p) - \bh(\ulbmuh - \bmIhPsik \bmu, q_h) 
		\\
		\lesssim \HHOnormhPsi{\ulbmvh}h^{k+1}\Brac{\nu \seminorm[\HS{k+2}(\Th)^d]{\bmu_r} + \seminorm[\HS{k+1}(\Th)]{p_r}}.
	\end{multline}
	Taking the supremum of \eqref{eq:error.estimate.proof.1} over all $(\ulbmvh, q_h)\in \bmXhPsiPhik$ with $\norm[X,h]{(\ulbmvh, q_h)} = 1$, and noting that $\HHOnormhPsi{\ulbmvh} \le \nu^{-\frac12}\norm[X,h]{(\ulbmvh, q_h)}$, it is inferred that
	\begin{multline*}
		\sup_{(\ulbmvh, q_h) \in \bmXhPsiPhik : \norm[X,h]{(\ulbmvh, q_h)} = 1}\Brac{\nu \ah(\ulbmuh - \bmIhPsik \bmu, \ulbmvh) + \bh(\ulbmvh, p_h - \SigmahPhi{k} p) - \bh(\ulbmuh - \bmIhPsik \bmu, q_h)} 
		\nl
		\lesssim h^{k+1}\Brac{\nu^\frac12 \seminorm[\HS{k+2}(\Th)^d]{\bmu_r} + \nu^{-\frac12}\seminorm[\HS{k+1}(\Th)]{p_r}}.
	\end{multline*}
	The result follows by applying the $\inf$--$\sup$ condition \eqref{eq:inf.sup.stability} with $\ulbmwh = \ulbmuh - \bmIhPsik \bmu$ and $z_h = p_h - \SigmahPhi{k} p$.
\end{proof}	

\section{Stokes Flow Around Submerged Cylinders}\label{sec:cylinders}

Inspired by the work of \textcite{wagner.moes.ea:2001:extended}, this study aims to explore the behaviour of the Stokes equations around submerged cylinders. The objective is to derive an analytical understanding of local fluid behaviour around a cylinder and subsequently enrich the local HHO spaces with the analytic solution.

The problem at hand involves finding solutions to the following set of equations:
\begin{subequations}\label{eq:stokes.circle}
	\begin{alignat}{2}
		-\bmDelta \bmu + \nabla p &= \bm{0} \quad &\textrm{in }& \R^2\backslash B_R, \label{eq:stokes.circle.momentum}\\
		\DIV \bmu &= 0 \quad &\textrm{in }& \R^2\backslash B_R, \label{eq:stokes.circle.mass}
	\end{alignat}
	subject to the boundary conditions:
	\begin{equation}\label{eq:stokes.circle.bc.restrictive}
		\bmu = \bmU_\infty \quad \textrm{as } |\bmx|\to\infty \quad;\quad \bmu = \bm{0} \quad \textrm{on } \partial B_R, 
	\end{equation}
	where $B_R$ denotes the circle of radius $R > 0$ centred at the origin and $\bmU_\infty$ represents a constant velocity field.
\end{subequations}
However, there exist no solutions to equations \eqref{eq:stokes.circle}, an unexpected result initially observed by \textcite{stokes:1851:effect}. Consequently, less restrictive boundary conditions are considered, expressed as:
\begin{equation}\label{eq:stokes.circle.bc}
	\bmu = \bmU_\infty \quad \textrm{at } |\bmx|=R_{\infty} \quad;\quad \bmu = \bm{0} \quad \textrm{on } \partial B_R,
\end{equation}
for some $R_\infty \gg 1$. For simplicity, the velocity field $\bmU_{\infty}$ is assumed to be given by $U \bme_x$, where $U\in\R$ and $\bme_x$ represents the unit vector in the $x$-direction.

Consider a stream function $\psi$ such that, in polar coordinates,
\begin{equation}\label{eq:u.stream}
	\bmu(r, \theta) = -\frac{1}{r}\partial_\theta \psi(r, \theta) \bme_r + \partial_r \psi(r, \theta) \bme_\theta,
\end{equation}
where $\bme_r$ and $\bme_\theta$ denote the radial and tangential unit vectors, respectively. Substituting equation \eqref{eq:u.stream} into \eqref{eq:stokes.circle.momentum} and taking the $\CURL$, the following equation is obtained:
\begin{subequations}\label{eq:stokes.circle.stream}
	\begin{equation}\label{eq:stokes.stream.circle}
		\Delta^2 \psi = 0,
	\end{equation}
	subject to the boundary conditions
	\begin{align}\label{eq:stokes.stream.circle.bc} 	
		\partial_{r}\psi(R_\infty, \theta) \eq -U\sin(\theta),\nl		
		\partial_{\theta}\psi(R_\infty, \theta) \eq -UR_{\infty}\cos(\theta), \nl 
		\partial_r\psi(R, \theta) \eq 0, \nl 
		\partial_\theta\psi(R, \theta) \eq 0.
	\end{align}
\end{subequations}
Due to the boundary conditions at $R_{\infty}$, it is expected that the solution will take the form:
\[
\psi(r, \theta) = \zeta(r) \sin(\theta).
\]
Upon substituting this back into equation \eqref{eq:stokes.circle.stream}, a solution for $\zeta(r)$ is obtained as follows:
\[
\zeta(r) = \frac{(R^2+R_{\infty}^2)U}{R^2-R_{\infty}^2+(R^2+R_{\infty}^2)\ln(\frac{R_{\infty}}{R})}\BRAC{\frac{r^2-R^2}{2r}-r\ln\Brac{\frac{r}{R}}+\frac{(r^2-R^2)^2}{2r(R^2+R_{\infty}^2)}}.
\]
A brief asymptotic analysis reveals that
\[
\zeta(r) \sim \frac{U}{\ln(R_{\infty})}\zeta_1(r) + \frac{U}{R_{\infty}^2\ln(R_{\infty})}\zeta_2(r) \qquad \textrm{as } R_{\infty} \to \infty,
\]
where
\[
\zeta_1(r) = \frac{r^2-R^2}{2r}-r\ln\Brac{\frac{r}{R}}\quad\textrm{and}\quad
\zeta_2(r) = \frac{(r^2-R^2)^2}{2r}.
\]
The two solutions $\zeta_1$ and $\zeta_2$ correspond to the two solutions considered by \textcite[Equations A12, A13]{wagner.moes.ea:2001:extended}. However, as the multiplicative constant of $\zeta_1$ dominates that of $\zeta_2$ as $R_{\infty}\to\infty$, the focus here will be solely on the solution $\zeta_1$. Therefore, the stream function considered is
\[
\psi(r, \theta) = \zeta_1(r) \sin(\theta).
\]
The corresponding velocity and pressure fields are given by the following expressions:

\begin{equation}\label{eq:cylinder.velocity}
	\bmu = \frac{R^2-r^2}{2r^4}\Brac{(x^2-y^2)\bme_x + 2xy \bme_y} + \ln\BRAC{\frac{r}{R}}\bme_x,
\end{equation}
and
\begin{equation}\label{eq:cylinder.pressure}
	p = C - \frac{2x}{r^2},
\end{equation}
where $C$ is an arbitrary constant.

\subsection{Single Submerged Cylinder with Manufactured Solution}\label{sec:cylinders.testA}

Consider the domain $\Omega = (0,1)^2 \setminus B_R(\bmx_0)$, where $B_R(\bmx_0)$ represents a ball of radius $R$ centred at $\bmx_0 = (\frac12, \frac12)$.

Let $\hat{\bmu}$ and $\hat{p}$ denote the asymptotic solutions \eqref{eq:cylinder.velocity} and \eqref{eq:cylinder.pressure} around the cylinder $B_R(\bmx_0)$, with the additive constant in $\hat{p}$ set to $0$.

An exact solution to the Stokes problem \eqref{eq:stokes} with viscosity $\nu=1$ is introduced as follows:
\[
\bmu = \hat{\bmu} + \sin(\pi x) \sin(\pi y)\left(\sin(\pi x) \cos(\pi y)\bme_x - \cos(\pi x) \sin(\pi y)\bme_y\right),
\]
and
\[
p = \hat{p} + (x - \frac12)(y - \frac12)^2.
\]

The enrichment spaces are defined as follows:
\[
\bmPsi(T) = \SPAN{\delta_\gamma(T) \hat{\bmu}}, \quad \Phi(T) = \SPAN{\delta_\gamma(T) \hat{p}},
\]
where $\delta_\gamma:\Th\to\{0,1\}$ is a cutoff function taking the value $1$ on elements with a centre of mass within a distance $\gamma>0$ of the cylinder $B_R(\bmx_0)$, and $0$ otherwise.

While both $\hat{\bmu}$ and $\hat{p}$ are smooth in the domain $\Omega$, they exhibit singularities at the centre of the cylinder $B_R(\bmx_0)$. Thus, as the cylinder radius $R$ approaches $0$, $\hat{\bmu}$ and $\hat{p}$ are expected to become nearly singular. Consequently, the enriched hybrid high-order method is anticipated to outperform the classical HHO method as $R\to 0$.

The enriched HHO method described in Section \ref{sec:model} is implemented using the open-source C++ library \texttt{PolyMesh} \parencite{polymesh}. A sequence of uniform Cartesian meshes is considered, and the cylinder $B_R(\bmx_0)$ is `cut out' of the mesh. The basis functions on curved edges are generated by first creating a spanning set, and removing the linearly dependent basis functions using the \texttt{FullPivLU} class found in the Eigen library, with documentation available at \url{https://eigen.tuxfamily.org/dox/ classEigen_1_1FullPivLU.html}.

Mesh data for the sequences of meshes of the domain $\Omega$ with radii $R=0.1$ and $R=0.01$ is displayed in Tables \ref{table:cylinder.R01} and \ref{table:cylinder.R001}, respectively.

\begin{table}[!ht]
	\centering
	\pgfplotstableread{data/nonenriched_circle_withsource_r01_k0.dat}\loadedtable
	\pgfplotstabletypeset
	[
	columns={MeshTitle, MeshSize,NbCells,NbInternalEdges}, 
	columns/MeshTitle/.style={column name=Mesh \#},
	columns/MeshSize/.style={column name=\(h\),/pgf/number format/.cd,fixed,zerofill,precision=4},
	columns/NbCells/.style={column name=Nb. Elements},
	columns/NbInternalEdges/.style={column name=Nb. Internal Edges},
	every head row/.style={before row=\toprule,after row=\midrule},
	every last row/.style={after row=\bottomrule} 
	%		sci=false
	]\loadedtable
	\caption{Mesh data of the domain $\Omega$ with $R = 0.1$}
	\label{table:cylinder.R01}
\end{table}
\begin{table}[!ht]
	\centering
	\pgfplotstableread{data/nonenriched_circle_withsource_r001_k0.dat}\loadedtable
	\pgfplotstabletypeset
	[
	columns={MeshTitle, MeshSize,NbCells,NbInternalEdges}, 
	columns/MeshTitle/.style={column name=Mesh \#},
	columns/MeshSize/.style={column name=\(h\),/pgf/number format/.cd,fixed,zerofill,precision=4},
	columns/NbCells/.style={column name=Nb. Elements},
	columns/NbInternalEdges/.style={column name=Nb. Internal Edges},
	every head row/.style={before row=\toprule,after row=\midrule},
	every last row/.style={after row=\bottomrule} 
	%		sci=false
	]\loadedtable
	\caption{Mesh data of the domain $\Omega$ with $R = 0.01$}
	\label{table:cylinder.R001}
\end{table}

In Figure \ref{fig:meshes.cylinder}, Mesh 1 is plotted for the two radii $R=0.1$ and $R=0.01$ considered.

\begin{figure}[!ht]
	\centering
	\includegraphics[width=0.4\textwidth]{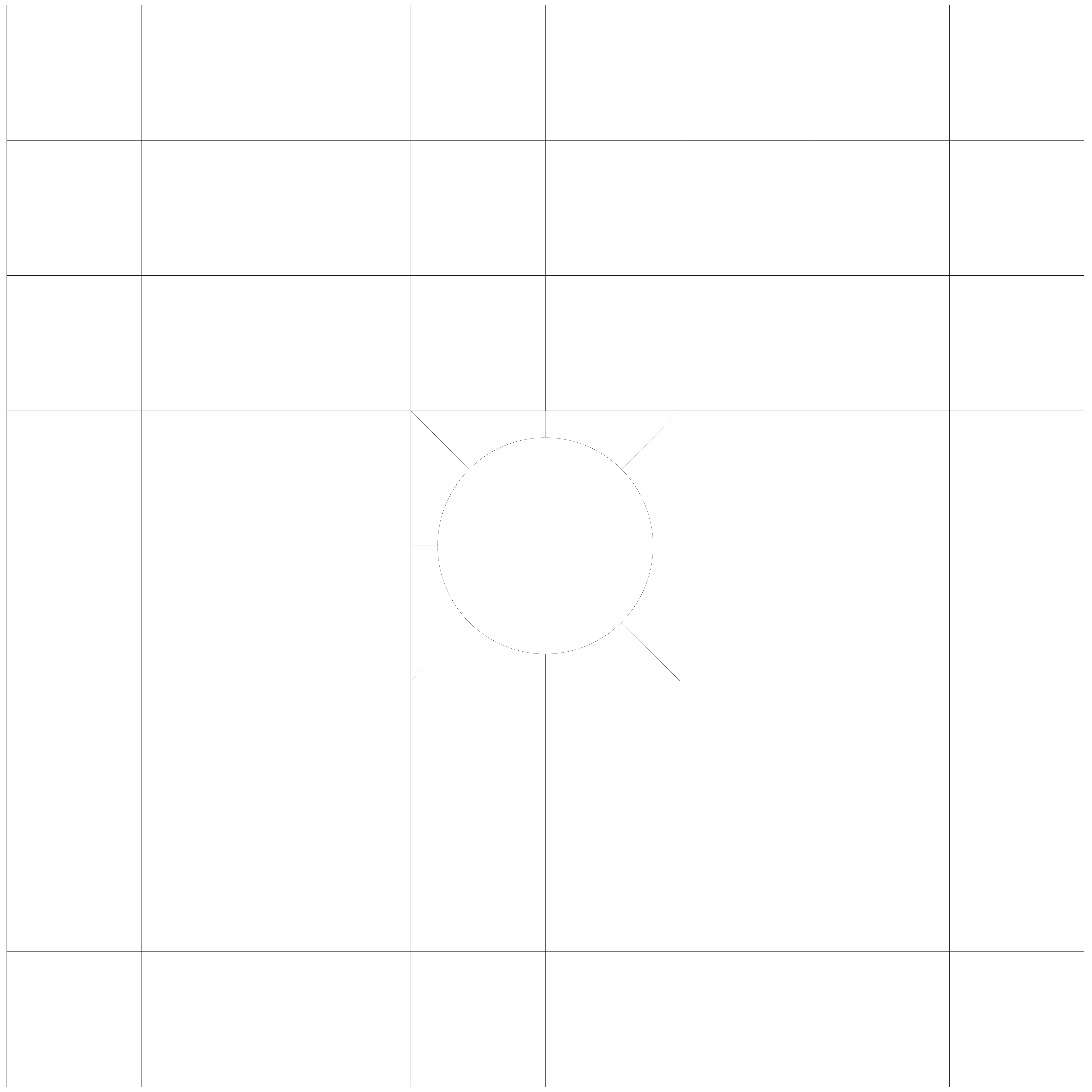}\hspace{0.5cm}
	\includegraphics[width=0.4\textwidth]{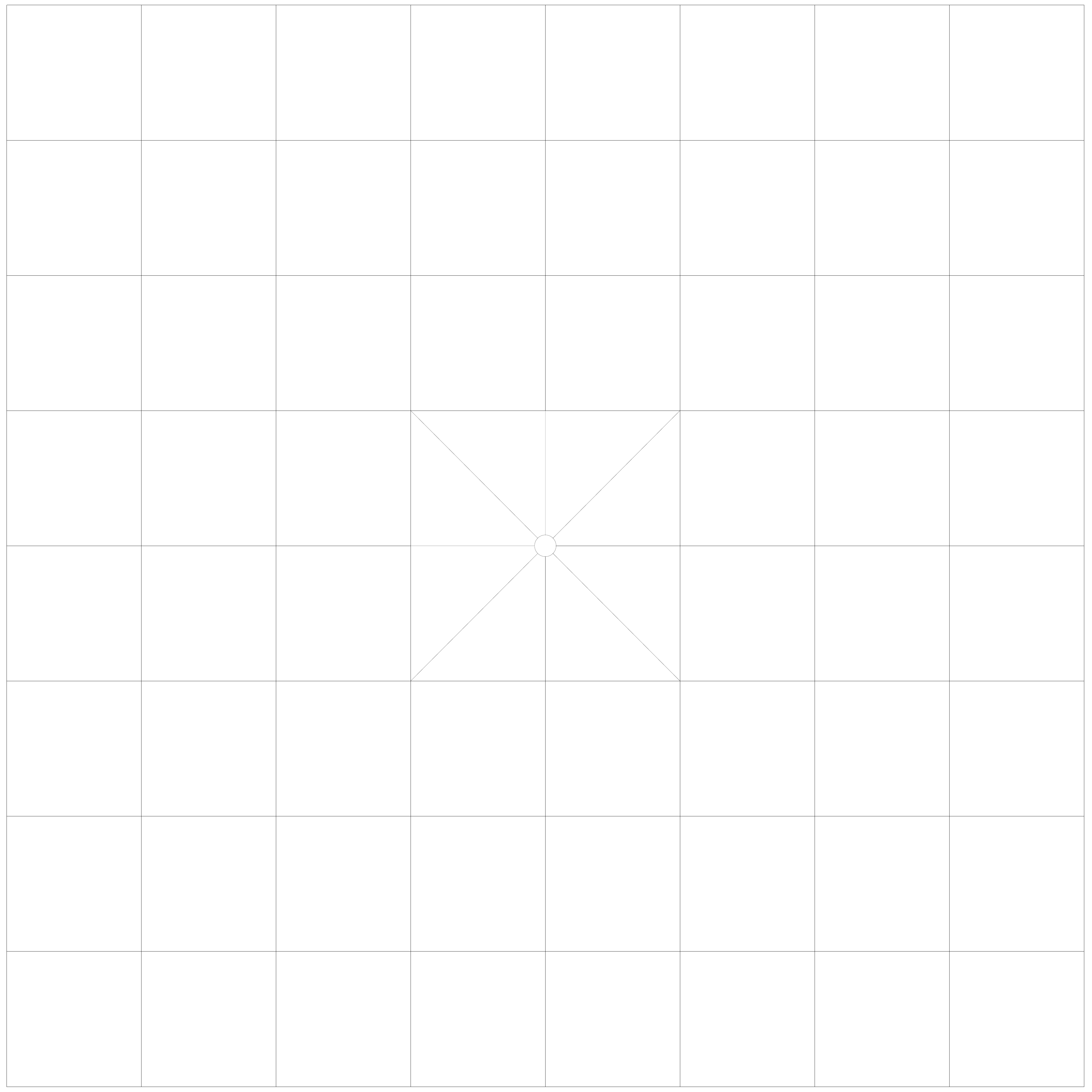} 
	\caption{Mesh 1 with $R=0.1$ (left) and $R=0.01$ (right)}
	\label{fig:meshes.cylinder}
\end{figure}

To measure the accuracy of the scheme, the following relative errors are defined:
\[
E_{0,h} = \frac{\norm[\Omega]{\bmrhPsi{k+1}\ulbmuh - \bmPi_{h, \bmPsi}^{k+1}\bmu}}{\norm[\Omega]{\bmPi_{h, \bmPsi}^{k+1}\bmu}},
\]
where $\bmPi_{h, \bmPsi}^{k+1}|_T=\bmPi_{T, \bmPsi}^{k+1}$ and $\bmPi_{T, \bmPsi}^{k+1}$ denotes the $L^2$-orthogonal projector on $\bmPOLYPsi{k+1}(T)$, 
\[
E_{\rma,h} = \frac{\ah(\ulbmuh - \bmIhPsi{k}\bmu, \ulbmuh - \bmIhPsi{k}\bmu)^\frac12}{\ah(\bmIhPsi{k}\bmu, \bmIhPsi{k}\bmu)^\frac12}\quad\textrm{and}\quad
%	\]
%	\[
E_{p,h} = \frac{\norm[\Omega]{p_h - \Pi_{h, \Phi}^{k}p}}{\norm[\Omega]{\Pi_{h, \Phi}^{k}p}}.
\]

As the enrichment of the HHO scheme increases the number of global degrees of freedom, it is most appropriate to test convergence rates against the degrees of freedom. The discrete Stokes system \eqref{eq:discrete} is amenable to static condensation as described by \textcite{di-pietro.ern.ea:2016:discontinuous}. After performing static condensation, the only globally coupled degrees of freedom remaining are the velocity unknowns on each internal face, and one pressure unknown on each element. As such, the number of global degrees of freedom is given by
\[
\textrm{DOFs} = \CARD{\Th} + \sum_{F\in\Fhi} \dim{\bmPOLYn{k}(F)}.
\]

In Figure \ref{fig:circle_test_R01}, a classical (non-enriched) HHO scheme as well as two enriched schemes with local enrichment $\gamma=0.1$ and $\gamma=0.2$ are tested on the sequence of meshes with $R=0.1$. It is evident that the enriched schemes achieve better approximation rates against the degrees of freedom.

\begin{figure}[!ht]
	\centering
	\ref{stokes.rtest.legend}
	\vspace{0.5cm}\\
	\subcaptionbox{$E_{0,h}$ vs DOFs, $k=0$}
	{
		\begin{tikzpicture}[scale=0.52]
			\begin{loglogaxis}[ legend columns=3, legend to name=stokes.rtest.legend ]
				\legend{Non-enriched, Locally enriched ($\gamma = 0.1$), Locally enriched ($\gamma = 0.2$)};
				\addplot table[x=DOFs,y=L2Error] {data/nonenriched_circle_withsource_r01_k0.dat};
				\addplot table[x=DOFs,y=L2Error] {data/enriched_circle_local01_withsource_r01_k0.dat};
				\addplot table[x=DOFs,y=L2Error] {data/enriched_circle_local02_withsource_r01_k0.dat};
				\reverseLogLogSlopeTriangle{0.50}{0.4}{0.1}{1}{black};
			\end{loglogaxis}
		\end{tikzpicture}
	}
	\subcaptionbox{$E_{\a,h}$ vs DOFs, $k=0$}
	{
		\begin{tikzpicture}[scale=0.52]
			\begin{loglogaxis}
				\addplot table[x=DOFs,y=EnergyError] {data/nonenriched_circle_withsource_r01_k0.dat};
				\addplot table[x=DOFs,y=EnergyError] {data/enriched_circle_local01_withsource_r01_k0.dat};
				\addplot table[x=DOFs,y=EnergyError] {data/enriched_circle_local02_withsource_r01_k0.dat};
				\reverseLogLogSlopeTriangle{0.50}{0.4}{0.1}{0.5}{black};
			\end{loglogaxis}
		\end{tikzpicture}
	}
	\subcaptionbox{$E_{\rmp,h}$ vs DOFs, $k=0$}
	{
		\begin{tikzpicture}[scale=0.52]
			\begin{loglogaxis}
				\addplot table[x=DOFs,y=PressureError] {data/nonenriched_circle_withsource_r01_k0.dat};
				\addplot table[x=DOFs,y=PressureError] {data/enriched_circle_local01_withsource_r01_k0.dat};
				\addplot table[x=DOFs,y=PressureError] {data/enriched_circle_local02_withsource_r01_k0.dat};
				\reverseLogLogSlopeTriangle{0.50}{0.4}{0.1}{0.5}{black};
			\end{loglogaxis}
		\end{tikzpicture}
	}
	%	\caption{Error vs DOFs with source, \(k=0\), \(R = 0.1\)}
	%	\label{fig:circle_test_R01_k0}
	%\end{figure}
	%
	%
	%\begin{figure}[!ht]
	%\centering
	%\ref{stokes.rtest.legend}
	\vspace{0.5cm}\\
	\subcaptionbox{$E_{0,h}$ vs DOFs, $k=1$}
	{
		\begin{tikzpicture}[scale=0.52]
			\begin{loglogaxis}
				\addplot table[x=DOFs,y=L2Error] {data/nonenriched_circle_withsource_r01_k1.dat};
				\addplot table[x=DOFs,y=L2Error] {data/enriched_circle_local01_withsource_r01_k1.dat};
				\addplot table[x=DOFs,y=L2Error] {data/enriched_circle_local02_withsource_r01_k1.dat};
				\reverseLogLogSlopeTriangle{0.50}{0.4}{0.1}{1.5}{black};
			\end{loglogaxis}
		\end{tikzpicture}
	}
	\subcaptionbox{$E_{\a,h}$ vs DOFs, $k=1$}
	{
		\begin{tikzpicture}[scale=0.52]
			\begin{loglogaxis}
				\addplot table[x=DOFs,y=EnergyError] {data/nonenriched_circle_withsource_r01_k1.dat};
				\addplot table[x=DOFs,y=EnergyError] {data/enriched_circle_local01_withsource_r01_k1.dat};
				\addplot table[x=DOFs,y=EnergyError] {data/enriched_circle_local02_withsource_r01_k1.dat};
				\reverseLogLogSlopeTriangle{0.50}{0.4}{0.1}{1}{black};
			\end{loglogaxis}
		\end{tikzpicture}
	}
	\subcaptionbox{$E_{\rmp,h}$ vs DOFs, $k=1$}
	{
		\begin{tikzpicture}[scale=0.52]
			\begin{loglogaxis}
				\addplot table[x=DOFs,y=PressureError] {data/nonenriched_circle_withsource_r01_k1.dat};
				\addplot table[x=DOFs,y=PressureError] {data/enriched_circle_local01_withsource_r01_k1.dat};
				\addplot table[x=DOFs,y=PressureError] {data/enriched_circle_local02_withsource_r01_k1.dat};
				\reverseLogLogSlopeTriangle{0.50}{0.4}{0.1}{1}{black};
			\end{loglogaxis}
		\end{tikzpicture}
	}
	\caption{Error vs DOFs, \(R = 0.1\)}
	\label{fig:circle_test_R01}
\end{figure}
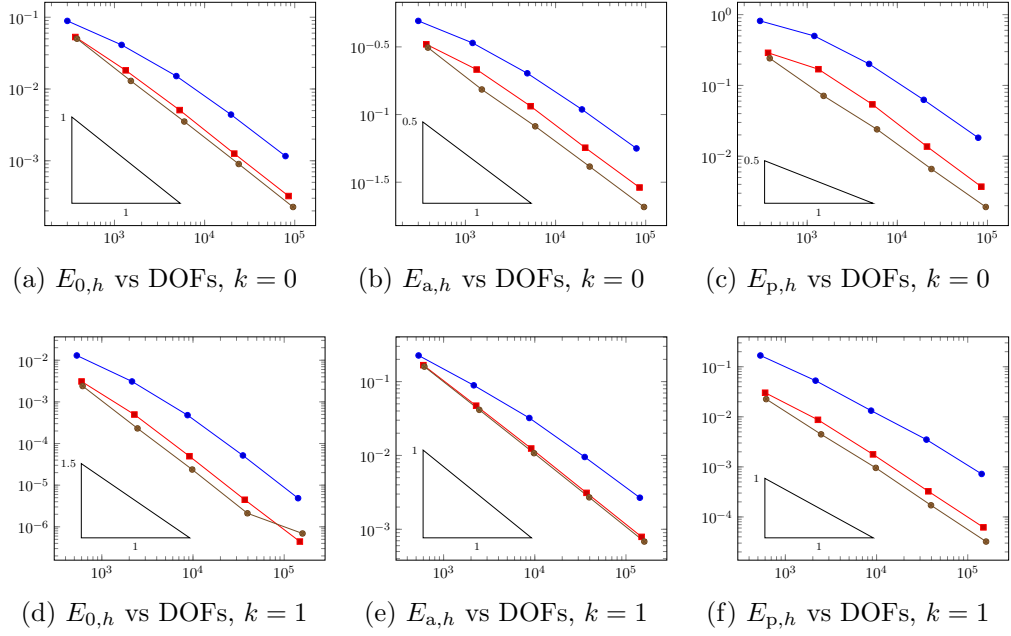

In Figure \ref{fig:circle_test_R001}, the same test is conducted on the mesh sequence with $R=0.01$. The enriched schemes perform significantly better than the non-enriched scheme. Moreover, the difference between the enriched and non-enriched scheme is far greater than when $R=0.1$. This matches the theoretical expectations as the velocity and pressure become increasingly steep close to the submerged cylinder as $R\to0$.

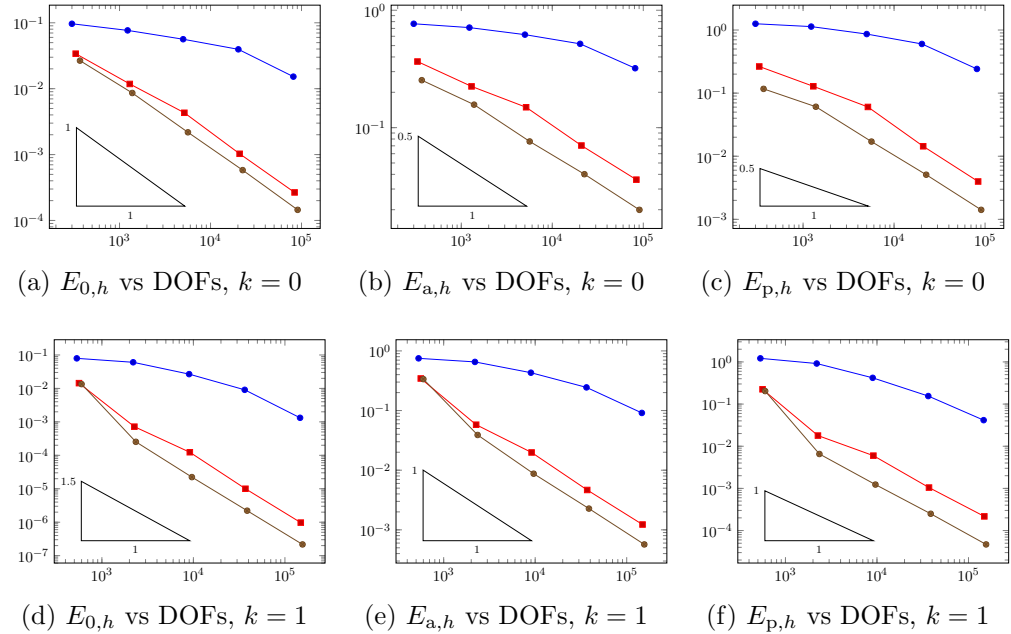
\begin{figure}[!ht]
	\centering
	\ref{stokes.rtest.legend}
	\vspace{0.5cm}\\
	\subcaptionbox{$E_{0,h}$ vs DOFs, \(k=0\)}
	{
		\begin{tikzpicture}[scale=0.52]
			\begin{loglogaxis}
				\addplot table[x=DOFs,y=L2Error] {data/nonenriched_circle_withsource_r001_k0.dat};
				\addplot table[x=DOFs,y=L2Error] {data/enriched_circle_local01_withsource_r001_k0.dat};
				\addplot table[x=DOFs,y=L2Error] {data/enriched_circle_local02_withsource_r001_k0.dat};
				\reverseLogLogSlopeTriangle{0.50}{0.4}{0.1}{1}{black};
			\end{loglogaxis}
		\end{tikzpicture}
	}
	\subcaptionbox{$E_{\a,h}$ vs DOFs, \(k=0\)}
	{
		\begin{tikzpicture}[scale=0.52]
			\begin{loglogaxis}
				\addplot table[x=DOFs,y=EnergyError] {data/nonenriched_circle_withsource_r001_k0.dat};
				\addplot table[x=DOFs,y=EnergyError] {data/enriched_circle_local01_withsource_r001_k0.dat};
				\addplot table[x=DOFs,y=EnergyError] {data/enriched_circle_local02_withsource_r001_k0.dat};
				\reverseLogLogSlopeTriangle{0.50}{0.4}{0.1}{0.5}{black};
			\end{loglogaxis}
		\end{tikzpicture}
	}
	\subcaptionbox{$E_{\rmp,h}$ vs DOFs, \(k=0\)}
	{
		\begin{tikzpicture}[scale=0.52]
			\begin{loglogaxis}
				\addplot table[x=DOFs,y=PressureError] {data/nonenriched_circle_withsource_r001_k0.dat};
				\addplot table[x=DOFs,y=PressureError] {data/enriched_circle_local01_withsource_r001_k0.dat};
				\addplot table[x=DOFs,y=PressureError] {data/enriched_circle_local02_withsource_r001_k0.dat};
				\reverseLogLogSlopeTriangle{0.50}{0.4}{0.1}{0.5}{black};
			\end{loglogaxis}
		\end{tikzpicture}
	}
	%	\caption{Error vs DOFs with source, \(k=0\), \(R = 0.01\)}
	%	\label{fig:circle_test_R001_k0}
	%\end{figure}
	%
	%
	%\begin{figure}[!ht]
	%	\centering
	%	\ref{stokes.rtest.legend}
	\vspace{0.5cm}\\
	\subcaptionbox{$E_{0,h}$ vs DOFs, \(k=1\)}
	{
		\begin{tikzpicture}[scale=0.52]
			\begin{loglogaxis}
				\addplot table[x=DOFs,y=L2Error] {data/nonenriched_circle_withsource_r001_k1.dat};
				\addplot table[x=DOFs,y=L2Error] {data/enriched_circle_local01_withsource_r001_k1.dat};
				\addplot table[x=DOFs,y=L2Error] {data/enriched_circle_local02_withsource_r001_k1.dat};
				\reverseLogLogSlopeTriangle{0.50}{0.4}{0.1}{1.5}{black};
			\end{loglogaxis}
		\end{tikzpicture}
	}
	\subcaptionbox{$E_{\a,h}$ vs DOFs, \(k=1\)}
	{
		\begin{tikzpicture}[scale=0.52]
			\begin{loglogaxis}
				\addplot table[x=DOFs,y=EnergyError] {data/nonenriched_circle_withsource_r001_k1.dat};
				\addplot table[x=DOFs,y=EnergyError] {data/enriched_circle_local01_withsource_r001_k1.dat};
				\addplot table[x=DOFs,y=EnergyError] {data/enriched_circle_local02_withsource_r001_k1.dat};
				\reverseLogLogSlopeTriangle{0.50}{0.4}{0.1}{1}{black};
			\end{loglogaxis}
		\end{tikzpicture}
	}
	\subcaptionbox{$E_{\rmp,h}$ vs DOFs, \(k=1\)}
	{
		\begin{tikzpicture}[scale=0.52]
			\begin{loglogaxis}
				\addplot table[x=DOFs,y=PressureError] {data/nonenriched_circle_withsource_r001_k1.dat};
				\addplot table[x=DOFs,y=PressureError] {data/enriched_circle_local01_withsource_r001_k1.dat};
				\addplot table[x=DOFs,y=PressureError] {data/enriched_circle_local02_withsource_r001_k1.dat};
				\reverseLogLogSlopeTriangle{0.50}{0.4}{0.1}{1}{black};
			\end{loglogaxis}
		\end{tikzpicture}
	}
	\caption{Error vs DOFs, \(R = 0.01\)}
	\label{fig:circle_test_R001}
\end{figure}

\subsection{Multiple Submerged Cylinders with Far-Field Boundary Conditions}\label{sec:cylinders.testB}

This section investigates the computational domain $\Omega = (0,1)^2 \backslash \bigcup_i B_{R_i}(\bmx_i)$, where four cylinders $B_{R_i}(\bmx_i)$ are positioned with centres and radii given by:	
\begin{alignat*}{3}
	\bmx_1 & = (0.4, 0.275), & \quad & R_1 & = 0.01, \\
	\bmx_2 & = (0.2, 0.25),  & \quad & R_2 & = 0.02, \\
	\bmx_3 & = (0.7, 0.5),   & \quad & R_3 & = 0.075, \\
	\bmx_4 & = (0.3, 0.75),  & \quad & R_4 & = 0.015.
\end{alignat*}

The far-field condition $\bmu = \bme_\bmx$ is imposed on the boundary of the square $(0,1)^2$, while the no-slip condition $\bmu = \bm{0}$ is enforced on the boundary of each cylinder.

A sequence of Cartesian meshes is considered, each with the cylinders `cut out' of the mesh. The mesh details are summarised in Table \ref{table:4circles.data}. Figure \ref{fig:cell.enrichment.plot} visualises an example mesh and two local enrichment schemes with $\gamma = 0.1$ and $\gamma = 0.2$, illustrating the number of enrichment functions on each element.

\begin{table}[!ht]
	\centering
	\pgfplotstableread{data/nonenriched_4circles_k0.dat}\loadedtable
	\pgfplotstabletypeset
	[
	columns={MeshTitle,MeshSize,NbCells,NbInternalEdges}, 
	columns/MeshTitle/.style={column name=Mesh \#},
	columns/MeshSize/.style={column name=\(h\),/pgf/number format/.cd,fixed,zerofill,precision=4},
	columns/NbCells/.style={column name=Nb. Elements},
	columns/NbInternalEdges/.style={column name=Nb. Internal Edges},
	every head row/.style={before row=\toprule,after row=\midrule},
	every last row/.style={after row=\bottomrule} 
	%		sci=false
	]\loadedtable
	\caption{Mesh data with multiple submerged cylinders}
	\label{table:4circles.data}
\end{table}

\begin{figure}[!ht]
	\centering
	\includegraphics[width=\linewidth]{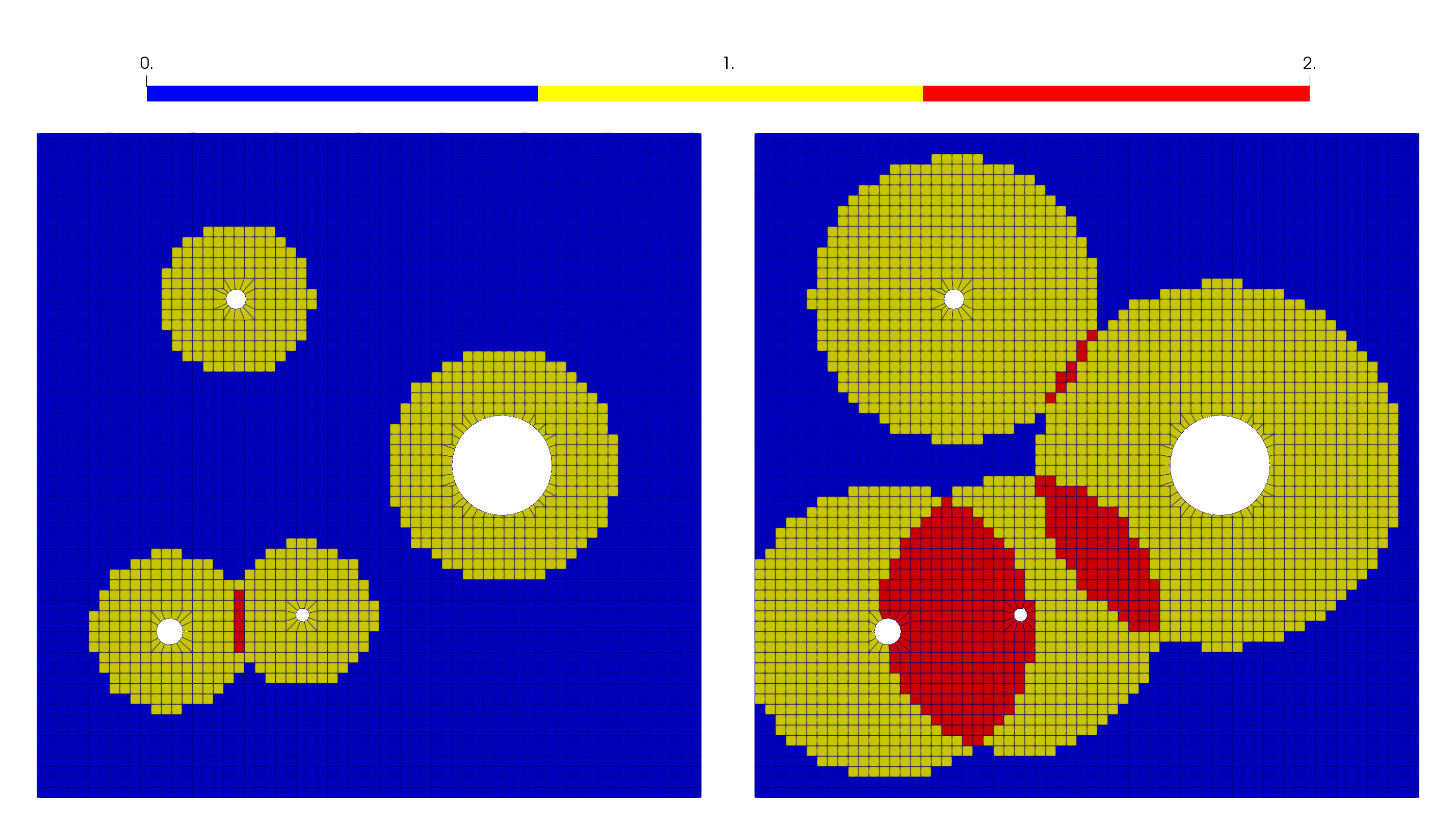}
	%\vspace{-1.5cm}
	%	\caption{Local enrichment scheme on Mesh \#4 with $\gamma = 0.2$. Colour indicates the dimension of the enrichment space on each element.}
	\caption{Local enrichment schemes on Mesh 4. Colour indicates the dimension of the enrichment space on each element. Left panel: $\gamma = 0.1$. Right panel: $\gamma = 0.2$}
	\label{fig:cell.enrichment.plot}
\end{figure}

Since the exact solution to this problem remains unknown, simulations are performed on Mesh 5 with $k=1$ and $\gamma = 0.2$. The resulting discrete velocities and pressures are denoted as $\bmrhPsi{k+1}\ulbmuh = \bmu_h^*$ and $p_h = p_h^*$, serving as the reference velocity and pressure fields. Quantities of interest are:

\[
\norm[\Omega]{p_h^*} \approx 15.610024 \quad ; \quad \seminorm[\HONE(\Th)]{u_h^*} \approx 5.7772666.
\]

The following metrics are defined to assess the performance of the scheme on coarser meshes (both enriched and non-enriched):

\[
E_p = \Seminorm{\norm[\Omega]{p_h} - \norm[\Omega]{p_h^*}} \quad\textrm{and}\quad E_\bmu = \SEMINORM{\seminorm[\HONE(\Th)]{\bmrhPsi{k+1}\ulbmuh} - \seminorm[\HONE(\Th)]{u_h^*}}.
\]

As these measures are not errors, the focus is not on the rate of convergence but rather on observing stable convergence and investigating the differences between the two schemes.

In Figures \ref{fig:4circles_test_k0} and \ref{fig:4circles_test_k1}, the convergence results are presented for $E_\bmu$ and $E_p$ on the sequence of meshes described in Table \ref{table:4circles.data} with $k=0$ and $k=1$, respectively. It is evident that the enriched schemes outperform the non-enriched scheme. Notably, there is little distinction between the two local enrichment schemes with $\gamma = 0.1$ and $\gamma = 0.2$ when plotted against the degrees of freedom. In Figure \ref{fig:4circles_test_k1}, the last data point for the test with $\gamma = 0.2$ is omitted, as the quantities $E_\bmu$ and $E_p$ are trivially zero due to coinciding with the reference solution.

\begin{figure}[!ht]
	\centering
	\ref{stokes.rtest.legend}
	\vspace{0.5cm}\\
	\subcaptionbox{$E_\bmu$ vs DOFs}
	{
		\begin{tikzpicture}[scale=0.52]
			\begin{loglogaxis}
				\addplot table[x=DOFs,y=H1Error] {data/nonenriched_4circles_k0.dat};
				\addplot table[x=DOFs,y=H1Error] {data/enriched_4circles_local01_k0.dat};
				\addplot table[x=DOFs,y=H1Error] {data/enriched_4circles_local02_k0.dat};
			\end{loglogaxis}
		\end{tikzpicture}
	}
	\subcaptionbox{$E_p$ vs DOFs}
	{
		\begin{tikzpicture}[scale=0.52]
			\begin{loglogaxis}
				\addplot table[x=DOFs,y=PressureError] {data/nonenriched_4circles_k0.dat};
				\addplot table[x=DOFs,y=PressureError] {data/enriched_4circles_local01_k0.dat};
				\addplot table[x=DOFs,y=PressureError] {data/enriched_4circles_local02_k0.dat};
			\end{loglogaxis}
		\end{tikzpicture}
	}
	\caption{Error vs DOFs, \(k=0\)}
	\label{fig:4circles_test_k0}
\end{figure}
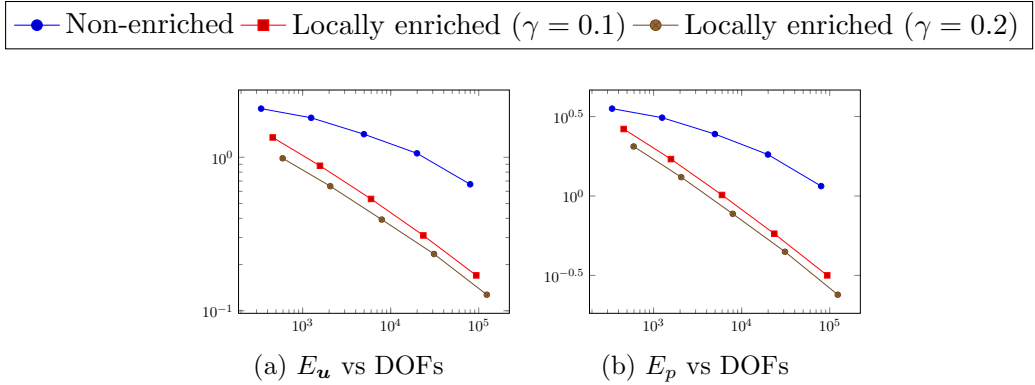

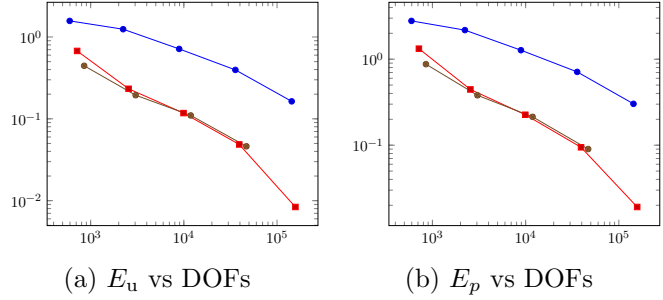
\begin{figure}[!ht]
	\centering
	\ref{stokes.rtest.legend}
	\vspace{0.5cm}\\
	\subcaptionbox{$E_\rmu$ vs DOFs}
	{
		\begin{tikzpicture}[scale=0.52]
			\begin{loglogaxis}
				\addplot table[x=DOFs,y=H1Error] {data/nonenriched_4circles_k1.dat};
				\addplot table[x=DOFs,y=H1Error] {data/enriched_4circles_local01_k1.dat};
				\addplot table[x=DOFs,y=H1Error] {data/enriched_4circles_local02_k1.dat};
			\end{loglogaxis}
		\end{tikzpicture}
	}
	\subcaptionbox{$E_p$ vs DOFs}
	{
		\begin{tikzpicture}[scale=0.52]
			\begin{loglogaxis}
				\addplot table[x=DOFs,y=PressureError] {data/nonenriched_4circles_k1.dat};
				\addplot table[x=DOFs,y=PressureError] {data/enriched_4circles_local01_k1.dat};
				\addplot table[x=DOFs,y=PressureError] {data/enriched_4circles_local02_k1.dat};
			\end{loglogaxis}
		\end{tikzpicture}
	}
	\caption{Error vs DOFs, \(k=1\)}
	\label{fig:4circles_test_k1}
\end{figure}

In Figure \ref{fig:stream.plots}, stream plots of the velocity fields on Mesh 4 with $k=1$ are displayed for both the non-enriched and locally enriched cases ($\gamma = 0.2$). While discerning differences between the two plots may not be straightforward, Figure \ref{fig:velocity.difference} illustrates the magnitude of the difference between the two schemes. Notably, the difference is most pronounced around the cylinders with the smallest radii, aligning with theoretical expectations. Furthermore, the magnitude of this difference can be as large as $0.13$, a significant value considering that the velocity field's largest magnitude is $1.45$, and the velocity is at its slowest around the cylinders.

\begin{figure}[!ht]
	\centering
	\includegraphics[width=\linewidth]{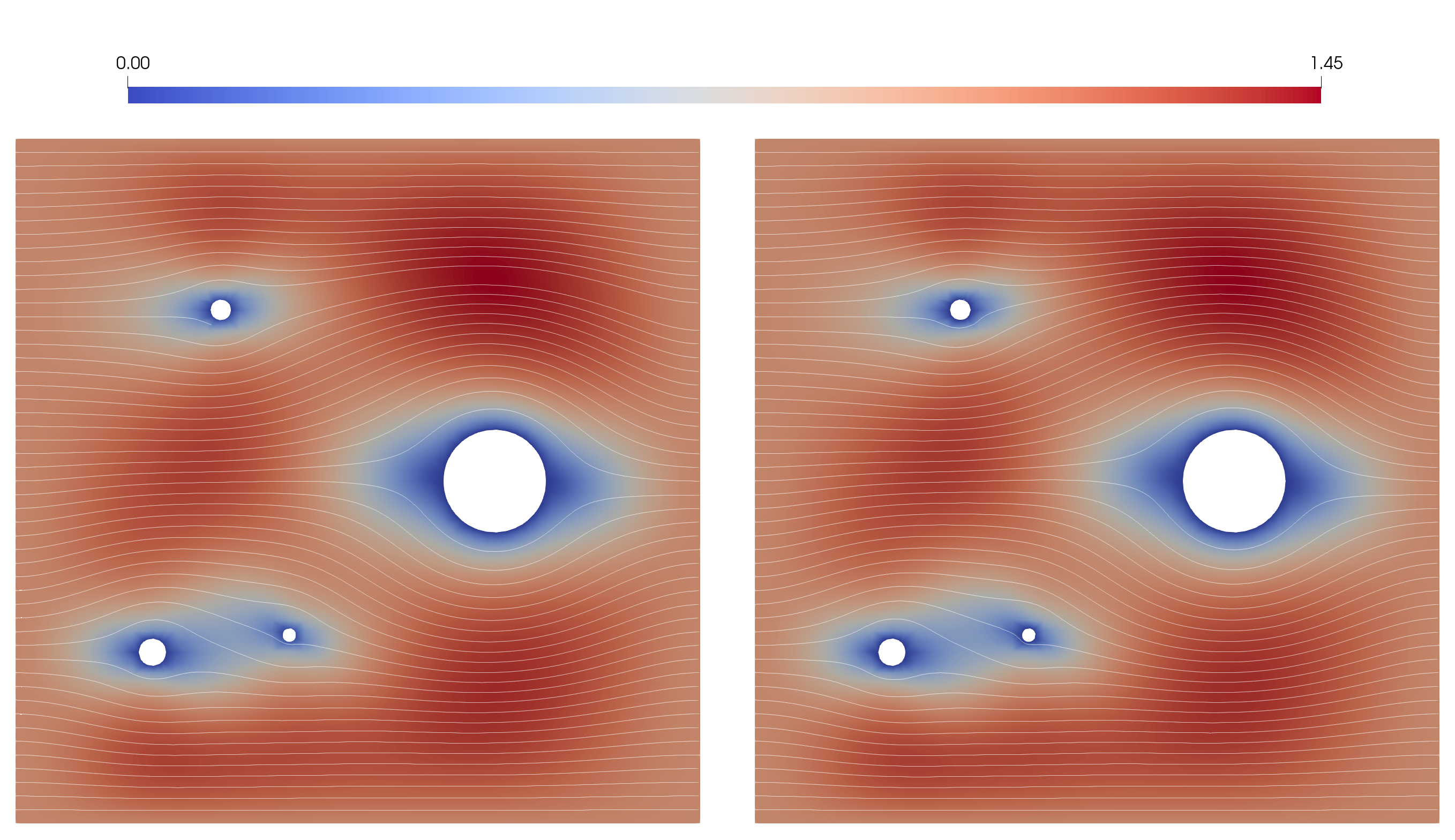}
	\caption{Stream plots of the velocity field on Mesh 4 with $k=1$. Left panel: non-enriched. Right panel: locally enriched ($\gamma=0.2$)}
	\label{fig:stream.plots}
\end{figure}

\begin{figure}[!ht]
	\centering
	\includegraphics[width=0.8\linewidth]{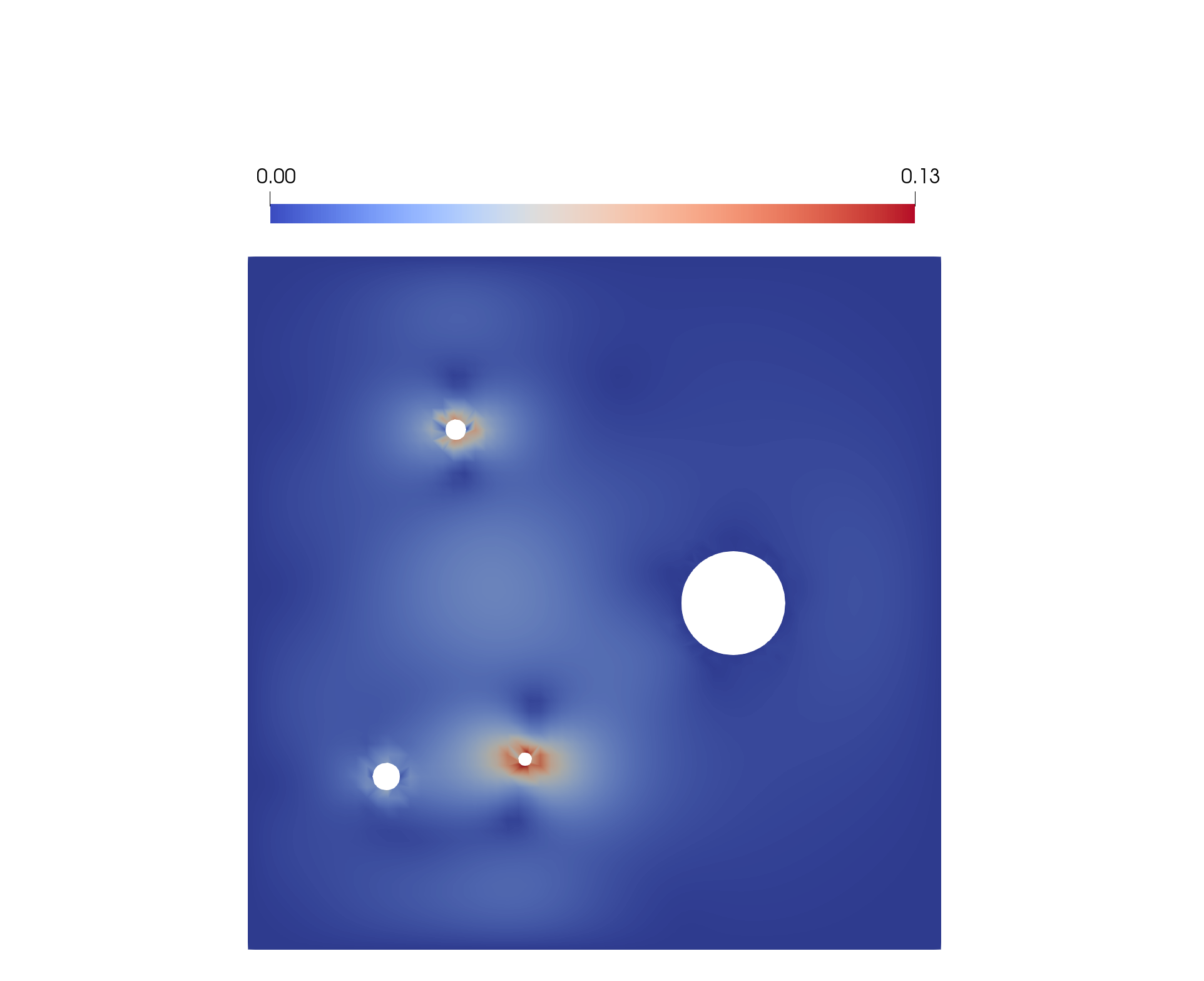}
	%\vspace{-1.5cm}
	\caption{Magnitude of the difference between the non-enriched velocity field and locally enriched ($\gamma=0.2$) velocity field on Mesh 4 with $k=1$}
	\label{fig:velocity.difference}
\end{figure}

Figure \ref{fig:pressure.plots} presents plots of the pressure fields on Mesh 4 with $k=1$ for both the non-enriched and locally enriched cases ($\gamma = 0.2$). The difference between the two plots is more pronounced than for the velocity field. Indeed, Figure \ref{fig:pressure.difference} showcases the magnitude of the difference between the two schemes, where the difference can reach a magnitude of similar order to the pressure field itself. Interestingly, the difference appears to be more concentrated around the cylinders than in the velocity field, which can be attributed to the stronger pressure singularity at the centre of each cylinder compared to the velocity.

\begin{figure}[!ht]
	\centering
	\includegraphics[width=\linewidth]{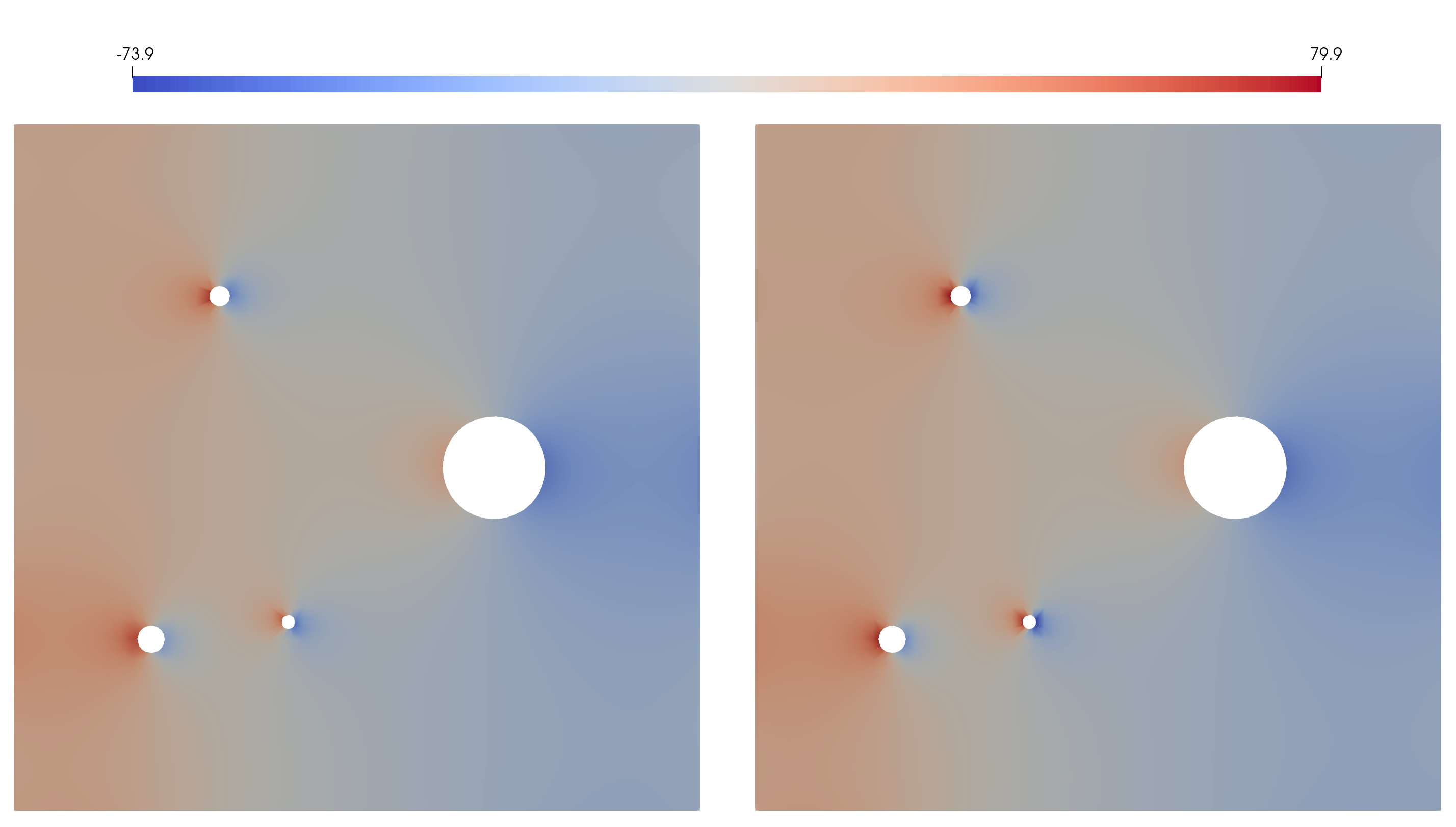}
	\caption{Plots of the pressure field on Mesh 4 with $k=1$. Left panel: non-enriched. Right panel: locally enriched ($\gamma=0.2$)}
	\label{fig:pressure.plots}
\end{figure}

\begin{figure}[!ht]
	\centering
	\includegraphics[width=0.7\linewidth]{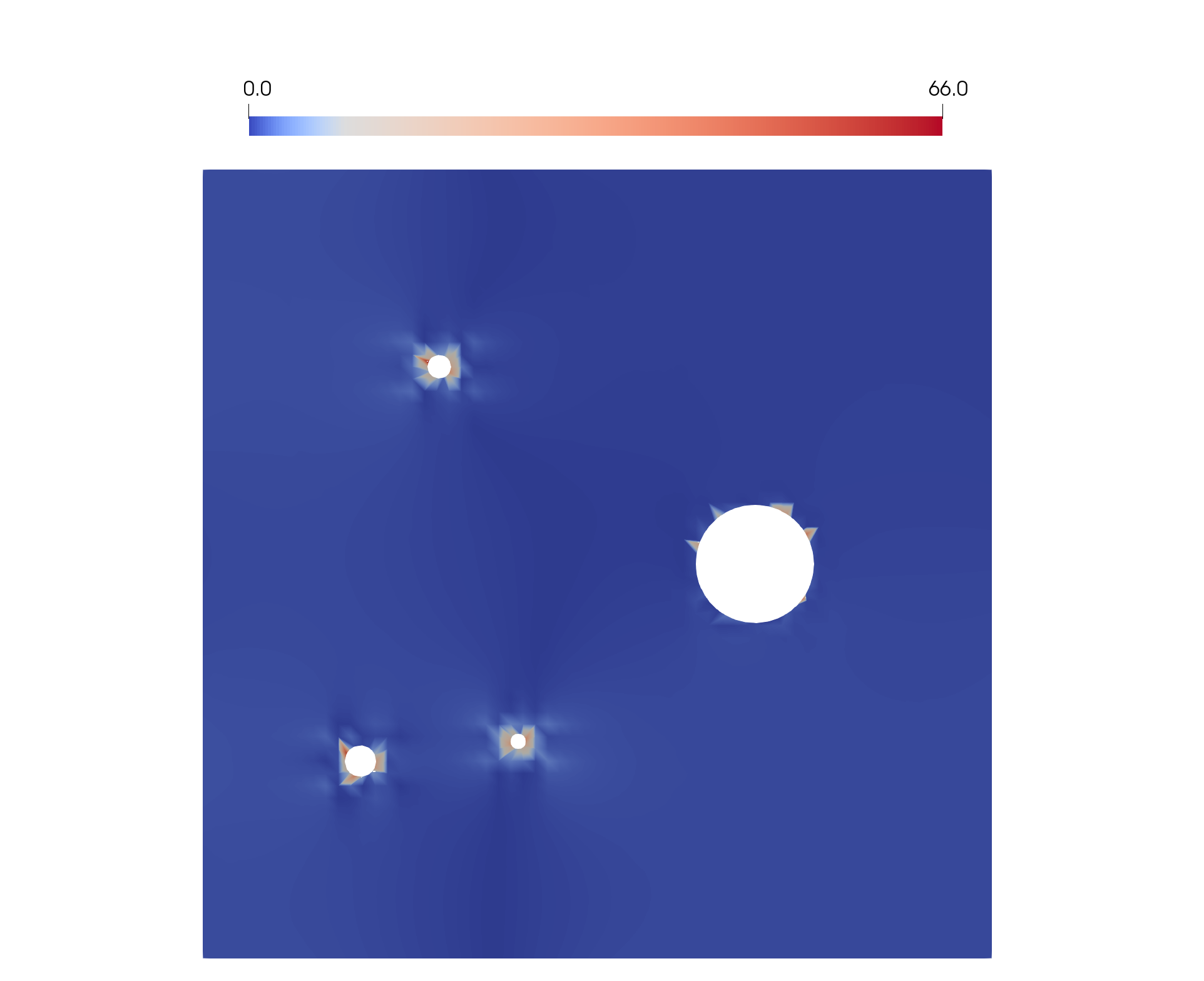}
	%\vspace{-1.5cm}
	\caption{Magnitude of the difference between the non-enriched pressure field and locally enriched ($\gamma=0.2$) pressure field on Mesh 4 with $k=1$}
	\label{fig:pressure.difference}
\end{figure}

\section*{Acknowledgements}

This research was supported by a Monash University Postgraduate Publications Award.

\printbibliography

\end{document}